%% file: siamart_NatDRM.tex
\documentclass[review,hidelinks,onefignum,onetabnum]{siamart251216}

\input{ex_shared}

\ifpdf
\hypersetup{
  pdftitle={Penalty-Free Natural Deep Ritz Method Based on de Rham Complex for High-Dimensional Dirichlet Boundary Value Problems},
  pdfauthor={J. Chen, X. Ji, H. Yu, and S. Zhang}
}
\fi

\usepackage{graphicx,epstopdf} 
\usepackage[caption=false]{subfig}
\usepackage{booktabs}

\begin{document}

\maketitle

\begin{abstract}
Deep neural networks show great promise for high-dimensional PDEs, yet enforcing essential boundary conditions remains challenging, especially as penalty parameters require problem-specific retuning with increasing dimensionality. In this work, we extend the Natural Deep Ritz Method (NatDRM) [H. Yu and S. Zhang, J. Comput. Phys., 537 (2025)] to a unified framework for all dimensions \(d \geq 2\) based on the de Rham complex and its penalty-free boundary decomposition: curl-type operators act on scalar potentials in 2D, vector potentials in 3D, and antisymmetric second-order tensor potentials in \(d \geq 4\), respectively. This method converts Dirichlet constraints into three coupled natural (Neumann-type) subproblems with corresponding Ritz-type losses, eliminating the need for a boundary penalty parameter \(\beta\). We derive dimension-unified discrete losses, lightweight boundary-based gauge-fixing regularizations to resolve curl-kernel non-uniqueness, and a joint training procedure; extensions to variable-coefficient elliptic and semilinear Poisson problems are formulated at the first subproblem level. Numerical experiments on smooth benchmarks up to 6D show that NatDRM, without any penalty tuning, matches or exceeds the accuracy of optimally tuned DRM and PINN in most cases. It converges stably in 6D where penalized DRM fails for most penalty values, and exhibits synchronous decay of interior and boundary errors, resolving the inherent imbalance of penalty-based methods.
\end{abstract}



\section{Introduction}

In recent years, the use of deep neural networks (DNNs) for solving partial differential equations (PDEs) has emerged as one of the most active research directions in computational science \cite{1,2,3,4,5}. Traditional numerical methods, such as finite element and finite difference schemes, while robust and theoretically well-founded, often suffer from the curse of dimensionality, wherein computational cost grows exponentially with the dimension of the problem domain \cite{6,7}. In contrast, deep learning-based approaches leverage the compositional structure of DNNs to mitigate this curse, enabling the treatment of high-dimensional PDEs previously deemed intractable. Foundational works, including the deep Ritz method (DRM) \cite{8}, physics-informed neural networks (PINNs) \cite{10}, deep Galerkin method (DGM) \cite{11}, and weak adversarial networks (WANs) \cite{12}, have achieved remarkable success in approximating solutions to both linear and nonlinear PDEs in high-dimensional spaces.

Despite these advances, a persistent challenge in the application of neural networks to boundary value problems is the rigorous enforcement of essential boundary conditions. In this paper, we focus on essential boundary value problems. As a representative example, consider the Poisson equation with Dirichlet boundary conditions defined on an arbitrary domain \(\Omega \subset \mathbb{R}^d\):
\begin{align}\label{1}
\left\{
\begin{aligned}
-\Delta u -f&= 0 && \text{in } \Omega, \\
u -g&= 0 && \text{on } \Gamma = \partial\Omega.
\end{aligned}
\right.
\end{align}
Unlike mesh-based methods where boundary conditions are enforced nodally, DNNs represent global trial functions. Consequently, imposing Dirichlet conditions often relies on penalty methods \cite{10,11}, distance-based conforming constructions \cite{13}, or Lagrange multiplier frameworks. DRM is a widely adopted penalty based approach that addresses this issue by minimizing the following objectives:
\[\mathcal{L}_{\text{DRM}}(u) = \sum_{x_j\in D}\left[\frac{1}{2}|\nabla u(x_j)|^2 - f(x_j)u(x_j)\right] + \beta \sum_{x_j\in D_\Gamma}(u(x_j) - g(x_j))^2,\]
where \(D\) and \(D_\Gamma\) denote the sets of training data points sampled in the interior of the domain and on the boundary, respectively, and \(\beta>0\) is a penalty parameter. PINNs also handle boundary conditions via a penalty strategy, with the loss function given by
\[\mathcal{L}_{\text{PINN}}(u) = \sum_{x_j\in D}|\Delta u(x_j) + f(x_j)|^2 + \beta \sum_{x_j\in D_\Gamma}(u(x_j) - g(x_j))^2.\] 
While effective, these strategies introduce practical complications: penalty methods require sensitive hyperparameter tuning and can induce numerical stiffness, whereas conforming methods demand careful construction of geometry-aware distance functions that become nontrivial in complex or high-dimensional domains \cite{14}. This challenge is particularly acute in high-dimensional settings, where the boundary geometry becomes complex and the surface-to-volume ratio of the domain changes drastically, making boundary sampling and penalty balancing increasingly delicate.

To address these limitations, a novel framework termed the natural deep Ritz method (NatDRM) was recently introduced for essential boundary value problems \cite{15}. The core innovation of NatDRM is that it transforms the original essential elliptic boundary value problem into a sequence of purely natural elliptic boundary value problems, rather than enforcing boundary conditions via penalty terms. All resulting subproblems remain well-posed elliptic problems with matching elliptic regularity, and their positive-definite structure preserves the inherent numerical stability of the original formulation. This approach, which is rooted in an adjoint formulation, leverages the mathematical structure of the de Rham complex and its dual \cite{16,17,18,19} to convert the discrepancy between the desired Dirichlet solution and its natural counterpart into auxiliary elliptic subproblems. It eliminates the need for penalty parameter tuning and avoids introducing artificial stiffness, thereby streamlining the optimization process and enhancing numerical robustness.

The main idea and core methodology of NatDRM have been fully illustrated and validated via two-dimensional model problems in \cite{15}, while its generalization to arbitrary dimension \(d\ge2\) calls for a consistent coordinate-free formulation. High-dimensional implementation introduces challenges beyond network width: boundary integrals, tensor-valued potentials with \(\binom{d}{2}\) components, and the coupling of three subproblems in training. Alternative boundary treatments in neural PDE solvers include distance-function ans\"atze \cite{13}, Nitsche-type weak boundary conditions, and Lagrange multipliers; finite element exterior calculus (FEEC) \cite{16,17} provides the operator calculus into which NatDRM embeds.

In this paper, we extend NatDRM to \(d\)-dimensional domains, give dimension-unified weak forms and discrete losses, and report numerical studies up to \(d=6\). Our focus is penalty-free essential boundary conditions rather than a complete survey of high-dimensional solvers (e.g., DGM \cite{11}, WAN \cite{12}).

This work makes three main contributions, establishing a unified penalty-free NatDRM for Dirichlet boundary value problems applicable to any dimension \(d \geq 2\), and resolving the core geometric and computational challenges introduced by dimensional scaling.

(1) Unified continuous formulation via de Rham complex geometry

We develop a dimension-agnostic natural boundary decomposition for Poisson-type problems, extending the 2D NatDRM theory through a consistent geometric framework that transitions seamlessly from scalar potentials (2D) to vector potentials (3D) and antisymmetric second-order tensor potentials (\(d \geq 4\)). We rigorously prove the decomposition correctness for all dimensions, addressing two critical theoretical barriers: the non-injectivity of the curl operator (which introduces a non-trivial kernel for the second subproblem in \(d \geq 3\)) and the geometric complexity jump from vector to tensor-valued potentials in \(d \geq 4\).

(2) Stable discrete NatDRM framework with gauge-fixing and general extensions

We derive a comprehensive suite of dimension-unified Ritz-type loss functions equipped with both boundary terms and tailored gauge-fixing regularizations. The proposed lightweight boundary-based mean regularization effectively resolves the non-uniqueness arising from the kernel of the curl operator, without introducing any additional tunable parameters. We design and analyze joint and alternating training strategies, and extend the framework to variable-coefficient elliptic and semilinear Poisson problems.

(3) High-dimensional validation with integration-accuracy tradeoff analysis

We present the first comprehensive numerical validation of NatDRM on 3D to 6D benchmark problems. A key focus is the optimal matching between integration precision and computational cost—a critical challenge in high dimensions where boundary sampling becomes sparse. Our results show that NatDRM achieves accuracy comparable to or better than optimally tuned DRM and PINN without any penalty tuning, maintains stable convergence in 6D where DRM fails for most penalty values, and exhibits synchronous decay of interior and boundary errors, eliminating the imbalance inherent in penalty-based methods.

The remainder of this paper is organized as follows. \Cref{sec:2} presents the natural formulation. \Cref{sec:3} details NatDRM losses and training. \Cref{sec:4} reports numerical experiments. \Cref{sec:5} concludes.

\section{A natural formulation for Poisson equations with Dirichlet boundary conditions}
\label{sec:2}

To demonstrate the main idea, we first focus on the Poisson equation \eqref{1} with Dirichlet boundary condition. We begin by establishing the necessary function space framework. $\Omega \subset \mathbb{R}^d$ stands for a simply connected domain with a boundary $\Gamma$, and we use $L^2(\Omega), H^1(\Omega), H_0^1(\Omega), H^{-1}(\Omega), H^{1/2}(\Gamma)\ \text{and}\ H^{-1/2}(\Gamma)$ for the standard Sobolev spaces. The variational formulation of \eqref{1} is to find $u \in H_g^1(\Omega) := \{ \omega \in H^1(\Omega) : \omega|_\Gamma =g \}$, such that
\begin{align}\label{2}
(\nabla u, \nabla v) = \langle f, v \rangle_{H^{-1}(\Omega) \times H^1_0(\Omega)},  \forall v \in H^1_0(\Omega).
\end{align}
Here, $\langle \cdot , \cdot \rangle_{H^{-1}(\Omega) \times H_0^1(\Omega)}$ denotes the duality pairing between $H^{-1}(\Omega)$ and $H_0^1(\Omega)$. Throughout this paper, we use $\langle \cdot , \cdot \rangle$ to represent dualities of various kinds, and we omit the subscripts when no ambiguity arises from the context.

In this section, we extend the two-dimensional NatDRM \cite{15} to three, four, and higher dimensions. The 2D theory is recalled without proof; 3D and \(d\ge4\) cases follow the same four-step reconstruction.

\subsection{Unified viewpoint and notation}
\label{sec:2-unified}

On \(\mathbb{R}^d\), the de Rham complex uses exterior derivative \(\mathrm{d}\) and codifferential \(\delta\), with \(\mathrm{d}^2=0\) and \(\delta^2=0\) \cite{16,17,18,19}. For a scalar potential \(u\), the gradient \(\nabla u\) is a 1-form; the discrepancy \(\nabla(u-\tilde u)\) is represented as \(\mathcal{C}\varphi\) for a curl-type potential \(\varphi\) whose tensor order depends on \(d\). Subproblem~1 is always a Neumann-compatible Poisson solve for \(\tilde u\); subproblem~2 recovers \(\varphi\) from boundary data of \(g-\tilde u\); subproblem~3 matches gradients and fixes the additive constant.

\begin{remark}[Relation to \cite{15}]\label{rem:rel15}
The four-step procedure in \cref{thm:2.2} and \cref{thm:2.3} generalizes \cite[Theorem~2.1]{15} from \(d=2\) to \(d\ge3\). New ingredients for \(d\ge3\) are the vector- or tensor-valued \(\varphi\), the \(H(\mathrm{curl},\Omega)\)-type space for \(\varphi\) when \(d\ge3\), and the corresponding boundary pairings in \eqref{13} and \eqref{19}.
\end{remark}

The de Rham complex framework admits a natural extension to weighted spaces for elliptic problems with non-constant coefficients. For a positive-definite symmetric matrix $\mathcal{A}(x)$, we can define a weighted inner product and corresponding weighted exterior derivative operators. This weighted structure underpins the variable-coefficient extension of NatDRM, as established in 2D in \cite{15}.

\begin{table}[htbp]
\centering
\caption{Dimension-specific operators and potentials.\label{tab:dimops}}
\small
\begin{tabular}{|c|c|c|}
\hline
\textbf{\(d\)} & \textbf{\(\varphi\)} & \textbf{\(\mathcal{C}=\mathrm{curl}\)}  \\
\hline
2 & scalar & \(\mathbb{R}^2\) vector \\
3 & vector & \(\mathbb{R}^3\) vector \\
\(\ge4\) & antisymmetric 2-tensor & \(\mathbb{R}^d\) vector \\
\hline
\end{tabular}
\end{table}

\subsection{Essential boundary value problems in 2D}

In the two-dimensional case, we denote the curl operators for scalar and vector functions by curl and rot, respectively. Namely, for a scalar function $w$, $\operatorname{curl} w(x, y) := [ \partial_y w , -\partial_x w ]^\top$, and for a vector function $v = [ v_1, v_2 ]^\top$, $\operatorname{rot} v := \partial_x v_2 - \partial_y v_1$. Note that $\operatorname{rot} \circ \operatorname{curl} = -\Delta$. For the 2D setting, the original essential boundary value problem is transformed into a sequence of boundary value problems as follows:

$\mathbf{Subproblem 1}$
\begin{equation}\label{3}
\left\{
\begin{aligned}
-\Delta \tilde{u} &= f, && \text{in } \Omega, \\
\frac{\partial \tilde{u}}{\partial \mathbf{n}} &= -\frac{1}{|\Gamma|} \int_{\Omega} f, && \text{on } \partial \Omega.
\end{aligned}
\right.
\end{equation}

$\mathbf{Subproblem 2}$
\begin{equation}\label{4}
\left\{
\begin{aligned}
-\Delta \varphi &= 0, && \text{in } \Omega, \\
\operatorname{curl} \varphi \cdot \mathbf{t} &= \partial_\mathbf{t} g - \partial_\mathbf{t}\tilde{u}, && \text{on } \partial \Omega.
\end{aligned}
\right.
\end{equation}

$\mathbf{Subproblem 3}$
\begin{equation}\label{5}
\left\{
\begin{aligned}
-\Delta u_c &= f, && \text{in } \Omega, \\
\frac{\partial u_c}{\partial \mathbf{n}} &= \partial_{\mathbf{n}} \tilde{u} - \partial_\mathbf{t} \varphi, && \text{on } \partial \Omega.
\end{aligned}
\right.
\end{equation}

Then $u_c$ is equal to $u$ up to a constant which can be fixed by the boundary condition.

\begin{lemma}[\cite{15}, Theorem 2.1.] Let $u$ be the solution of \eqref{2}, and $u^*$ be obtained by the four steps below:

1. Find $\tilde{u} \in H_\Gamma^1(\Omega) := \{ \omega \in H^1(\Omega) : \int_\Gamma \omega=0 \}$, such that
\begin{align}\label{6}
(\nabla \tilde{u}, \nabla v) = \langle \tilde{f}, v \rangle_{(H_{\Gamma}^{1}(\Omega))' \times H^1_\Gamma(\Omega)}, \ \forall v \in H^1_\Gamma(\Omega),
\end{align}
where $\tilde{f}$ is any extension of $f$ to $(H_{\Gamma}^{1}(\Omega))'$ such that $ \langle \tilde{f}, v \rangle =\langle f, v \rangle$ for $v \in H^1_0(\Omega)$.

2. Find a $ \varphi \in H^1(\Omega)$, such that
\begin{align}\label{7}
(\text{curl}\varphi , \text{curl} \psi)=\langle \partial_\mathbf{t}(g-\tilde{u}|_\Gamma),\psi\rangle_\Gamma,\ \forall \psi \in H^1(\Omega).
\end{align}
Here, $\langle \cdot , \cdot \rangle_\Gamma$ is a duality between $H^{-1/2}(\Gamma)$ and $H^{1/2}(\Gamma)$, which evaluates as the $L^2$ inner product on $\Gamma$ for sufficiently smooth functions.

3. Find a $ u_c \in H^1(\Omega)$, such that
\begin{align}\label{8}
(\nabla u_c, \nabla v)=(\nabla \tilde{u}-\text{curl}\varphi, \nabla v), \ \forall v \in H^1(\Omega).
\end{align}
4. Set $u^* =u_c-C$, with $C=\frac{1}{|\gamma|}\int_\gamma(u_c-g)$ for any $\gamma \subset \Gamma$ such that $|\gamma|\ne 0$.

Then $u^* =u$.
\end{lemma}

\subsection{Essential boundary value problems in 3D}

We use the div operator and curl operator for a vector function $\mathbf{v}=(v_1,v_2,v_3)^\top$,  $\text{div} \, \mathbf{v}=\nabla  \cdot (v_1,v_2,v_3) =\partial_1 v_1+\partial_2 v_2+\partial_3 v_3$ and 
$$\text{curl} \, \mathbf{v} = \nabla \times \mathbf{v} = \begin{pmatrix} \partial_2 v_3 - \partial_3 v_2 \\ \partial_3 v_1 - \partial_1 v_3 \\ \partial_1 v_2 - \partial_2 v_1 \end{pmatrix}.$$
Note that $\text{curl}\circ \nabla=0$ and $\text{div}\circ \text{curl}=0$.
For the 3D setting, the original essential boundary value problem is transformed into a sequence of boundary value problems as follows:

$\mathbf{Subproblem 1}$
\begin{equation}\label{9}
\left\{
\begin{aligned}
-\Delta \tilde{u} &= f, && \text{in } \Omega, \\
\frac{\partial \tilde{u}}{\partial \mathbf{n}} &= -\frac{1}{|\Gamma|} \int_{\Omega} f, && \text{on } \partial \Omega.
\end{aligned}
\right.
\end{equation}

$\mathbf{Subproblem 2}$
\begin{equation}\label{10}
\left\{
\begin{aligned}
\operatorname{curl} \operatorname{curl} \varphi &= 0, && \text{in } \Omega, \\
\operatorname{curl} \varphi \times \mathbf{n} &= \nabla_\Gamma (g - \tilde{u}|_\Gamma), && \text{on } \partial \Omega.
\end{aligned}
\right.
\end{equation}

$\mathbf{Subproblem 3}$
\begin{equation}\label{11}
\left\{
\begin{aligned}
-\Delta u_c &= f, && \text{in } \Omega, \\
\frac{\partial u_c}{\partial \mathbf{n}} &= \partial_{\mathbf{n}} \tilde{u} - \operatorname{curl} \varphi \cdot \mathbf{n}, && \text{on } \partial \Omega.
\end{aligned}
\right.
\end{equation}
Then $u_c$ is equal to $u$ up to a constant which can be fixed by the boundary condition.

\begin{theorem}\label{thm:2.2}
Let $u$ be the solution of \eqref{2}, and $u^*$ be obtained by the four steps below:

1. Find $\tilde{u} \in H_\Gamma^1(\Omega) := \{ \omega \in H^1(\Omega) : \int_\Gamma \omega=0 \}$, such that
\begin{align}\label{12}
(\nabla \tilde{u}, \nabla v) = \langle \tilde{f}, v \rangle_{(H_{\Gamma}^{1}(\Omega))' \times H^1_\Gamma(\Omega)}, \ \forall v \in H^1_\Gamma(\Omega),
\end{align}
where $\tilde{f}$ is any extension of $f$ to $(H_{\Gamma}^{1}(\Omega))'$ such that $ \langle \tilde{f}, v \rangle =\langle f, v \rangle$ for $v \in H^1_0(\Omega)$.

2. Find $\varphi \in \bigl(H^1(\Omega)\bigr)^3$ such that
\begin{align}\label{13}
(\text{curl}\varphi , \text{curl} \psi)=\langle \nabla_\Gamma(g-\tilde{u}|_\Gamma),\psi\rangle_\Gamma,\ \forall \psi \in \bigl(H^1(\Omega)\bigr)^3.
\end{align}
Here, $\langle \cdot , \cdot \rangle_\Gamma$ is a duality between $H^{-1/2}(\Gamma)$ and $H^{1/2}(\Gamma)$, which evaluates as the $L^2$ inner product on $\Gamma$ for sufficiently smooth functions.

3. Find a $ u_c \in H^1(\Omega)$, such that
\begin{align}\label{14}
(\nabla u_c, \nabla v)=(\nabla \tilde{u}-\text{curl}\varphi, \nabla v), \ \forall v \in H^1(\Omega).
\end{align}

4. Set $u^* =u_c-C$, with $C=\frac{1}{|\gamma|}\int_\gamma(u_c-g)$ for any $\gamma \subset \Gamma$ such that $|\gamma|\ne 0$.

Then $u^* =u$.
\end{theorem}

\begin{proof}
By \eqref{2} and \eqref{12}, $(\nabla u-\nabla \tilde{u}, \nabla v)=0$ for any $v \in H^1_0(\Omega)$ and it follows that $\nabla (u- \tilde{u}) = \text{curl} \varphi$ for some $ \varphi \in H^1(\Omega)$. Further, $\text{curl}(\nabla (u- \tilde{u}))=\text{curl}\ \text{curl}\varphi=0$. Therefore, for any $ \psi \in H^1(\Omega)$, we have 
\begin{align*}
(\text{curl}\,\varphi, \text{curl}\,\psi) &= \int_{\Omega} (\nabla \times \varphi) \cdot (\nabla \times \psi) d\Omega \\
&= \int_{\Omega} \psi \cdot (\nabla \times (\nabla \times \varphi)) d\Omega + \int_{\Omega} \nabla \cdot (\psi \times (\nabla \times \varphi)) d\Omega \\
&= (\text{curl}\,\text{curl}\,\varphi, \psi) + \int_{\partial \Omega} [\psi \times (\nabla \times \varphi)] \cdot \mathbf{n} d\Gamma \\
&= (\text{curl}\,\text{curl}\,\varphi, \psi) + \int_{\partial \Omega} (\nabla \times \varphi) \cdot (\mathbf{n} \times \psi) d\Gamma \\
&= (\text{curl}\,\text{curl}\,\varphi, \psi) + \langle \nabla \times \varphi, \mathbf{n} \times \psi \rangle_{\Gamma} \\
&= (\text{curl}\,\text{curl}\,\varphi, \psi) + \langle \text{curl}\,\varphi, \mathbf{n} \times \psi \rangle_{\Gamma} \\
&= 0 + \langle \text{curl}\,\varphi \times \mathbf{n}, \psi \rangle_{\Gamma}\\
&= \langle\nabla (u- \tilde{u})\times \mathbf{n}, \psi \rangle_{\Gamma}\\
&= \langle\nabla_\Gamma(u- \tilde{u}), \psi \rangle_{\Gamma}\\
&= \langle\nabla_\Gamma(g- \tilde{u}|_\Gamma), \psi \rangle_{\Gamma}\\
&= -\langle g- \tilde{u}|_\Gamma, \nabla_\Gamma\psi \rangle_{\Gamma},
\end{align*}
namely $\varphi$ satisfies \eqref{13}. Now we obtain by \eqref{14} that $\nabla u_c=\nabla u$. Then $u_c-u$ is a constant which can be corrected by Step 4 and finally, we are lead to that $u^*=u$. The proof is completed. 
\end{proof}

\subsection{Essential boundary value problems in 4D and higher dimensions}

We use $\text{curl}$ and $\text{rot}$ for the curl operators for antisymmetric second-order tensor and vector functions, respectively. 

Namely, for an antisymmetric second-order tensor 
$$
\mathbf{\omega} = \begin{pmatrix}
0 & \omega_{12} & \omega_{13} & \omega_{14} \\
-\omega_{12} & 0 & \omega_{23} & \omega_{24} \\
-\omega_{13} & -\omega_{23} & 0 & \omega_{34} \\
-\omega_{14} & -\omega_{24} & -\omega_{34} & 0
\end{pmatrix},
$$
$ \text{curl}\ \mathbf{\omega}(x_1,x_2,x_3,x_4):= -\sum_{j=1}^d \partial_j \omega_{ji}, $ and for a vector function $\mathbf{v}=(v_1,v_2,v_3,v_4)^\top$, 
$$\text{rot}\ \mathbf{v} = \begin{pmatrix} 0 & \frac{\partial v_2}{\partial x_1} - \frac{\partial v_1}{\partial x_2} & \frac{\partial v_3}{\partial x_1} - \frac{\partial v_1}{\partial x_3} & \frac{\partial v_4}{\partial x_1} - \frac{\partial v_1}{\partial x_4} \\ -\frac{\partial v_2}{\partial x_1} + \frac{\partial v_1}{\partial x_2} & 0 & \frac{\partial v_3}{\partial x_2} - \frac{\partial v_2}{\partial x_3} & \frac{\partial v_4}{\partial x_2} - \frac{\partial v_2}{\partial x_4} \\ -\frac{\partial v_3}{\partial x_1} + \frac{\partial v_1}{\partial x_3} & -\frac{\partial v_3}{\partial x_2} + \frac{\partial v_2}{\partial x_3} & 0 & \frac{\partial v_4}{\partial x_3} - \frac{\partial v_3}{\partial x_4} \\ -\frac{\partial v_4}{\partial x_1} + \frac{\partial v_1}{\partial x_4} & -\frac{\partial v_4}{\partial x_2} + \frac{\partial v_2}{\partial x_4} & -\frac{\partial v_4}{\partial x_3} + \frac{\partial v_3}{\partial x_4} & 0 \end{pmatrix}, $$ 
$\text{div} \, \mathbf{v}=\nabla  \cdot (v_1,v_2,v_3,v_4) =\partial_1 v_1+\partial_2 v_2+\partial_3 v_3+\partial_4 v_4$.  Note that $\text{rot}\circ \nabla=0$ and $\text{div}\circ \text{curl}=0$.

Basically, we adopt the procedure below to transfer the original essential boundary value problem to several boundary value problems as below:

$\mathbf{Subproblem \ 1}$
\begin{align}\label{15}
\left\{
\begin{aligned}
-\Delta \tilde{u} &= f && \text{in } \Omega, \\
\frac{\partial \tilde{u}}{\partial \mathbf{n}} &= -\frac{1}{|\Gamma|} \int_{\Omega} f \,\mathrm{d}\Omega && \text{on } \partial \Omega.
\end{aligned}
\right.
\end{align}

$\mathbf{Subproblem \ 2}$
\begin{align}\label{16}
\left\{
\begin{aligned}
\text{rot}\ \text{curl}\varphi &= 0 && \text{in } \Omega, \\
\langle  \mathbf{n} \wedge  \text{curl} \varphi, \psi \rangle_{\Gamma}&= \langle  g - \tilde{u} , \text{curl} \psi \cdot \mathbf{n} \rangle_{\Gamma}&& \text{on } \partial \Omega.
\end{aligned}
\right.
\end{align}

$\mathbf{Subproblem \ 3}$
\begin{align}\label{17}
\left\{
\begin{aligned}
-\Delta u_c &= f && \text{in } \Omega, \\
\frac{\partial u_c}{\partial \mathbf{n}} &= \partial_{\mathbf{n}} \tilde{u} - \text{curl} \varphi \cdot \mathbf{n} && \text{on } \partial \Omega.
\end{aligned}
\right.
\end{align}

Then $u_c$ is equal to $u$ up to a constant which can be fixed by the boundary condition.

\begin{theorem}\label{thm:2.3}  
Let $u$ be the solution of \eqref{2}, and $u^*$ be obtained by the four steps below:

1. Find $\tilde{u} \in H_\Gamma^1(\Omega) := \{ \omega \in H^1(\Omega) : \int_\Gamma \omega=0 \}$, such that
\begin{equation}\label{18}
    (\nabla \tilde{u}, \nabla v) = \langle \tilde{f}, v \rangle_{(H_{\Gamma}^{1}(\Omega))' \times H^1_\Gamma(\Omega)}, \ \forall v \in H^1_\Gamma(\Omega),
\end{equation}
where $\tilde{f}$ is any extension of $f$ to $(H_{\Gamma}^{1}(\Omega))'$ such that $ \langle \tilde{f}, v \rangle =\langle f, v \rangle$ for $v \in H^1_0(\Omega)$.

2. Find $\varphi$ in the space of antisymmetric matrix fields on $\Omega$ with $H(\mathrm{curl},\Omega)$-regularity (equivalently, a 2-form with square-integrable $\mathrm{curl}$), such that
\begin{equation}\label{19}
(\text{curl}\varphi , \text{curl} \psi)=\langle  g - \tilde{u} , \text{curl} \psi \cdot \mathbf{n} \rangle_{\Gamma},\ \forall \psi \text{ in the same class.}
\end{equation}
Here, $\langle \cdot, \cdot \rangle_\Gamma$ denotes the duality pairing between $H^{1/2}(\Gamma)$ and $H^{-1/2}(\Gamma)$, which reduces to the standard $L^2(\Gamma)$ inner product for sufficiently smooth functions. The well-posedness of this pairing is guaranteed by the trace theorem for $H(\text{div}, \Omega)$ spaces: for any 2-form $\psi \in H(d, \Omega; \Lambda^2)$ (i.e., $\psi \in H(\text{curl}, \Omega)$ in our notation), its curl $\text{curl}\,\psi = \star d\psi$ is a divergence-free 1-form belonging to $H(\text{div}, \Omega; \Lambda^1)$, and thus its normal component trace $\text{curl}\,\psi \cdot n|_\Gamma$ is a well-defined element of $H^{-1/2}(\Gamma)$ \cite{16,17}.

3. Find a $u_c \in H^1(\Omega)$, such that
\begin{equation}\label{20}
(\nabla u_c, \nabla v)=(\nabla \tilde{u}-\text{curl}\varphi, \nabla v), \ \forall v \in H^1(\Omega).
\end{equation}

4. Set $u^* =u_c-C$, with $C=\frac{1}{|\gamma|}\int_\gamma(u_c-g)$ for any $\gamma \subset \Gamma$ such that $|\gamma|\ne 0$.

Then $u^* =u$.
\end{theorem}

\begin{proof}
By \eqref{2} and \eqref{18}, $(\nabla u-\nabla \tilde{u}, \nabla v)=0$ for any $v \in H^1_0(\Omega)$ and it follows that $\nabla (u- \tilde{u}) = \text{curl} \varphi$ for some $ \varphi \in H^1(\Omega)$. Further applying the operator $\operatorname{rot}$ to both sides gives $\text{rot}(\nabla (u- \tilde{u}))=\text{rot}\ \text{curl}\varphi=0$. Now take any test tensor $\psi \in H(\operatorname{curl},\Omega)$. Taking the $L^2$ inner product of $\nabla (u- \tilde{u}) = \text{curl} \varphi$ with $\operatorname{curl} \psi$ and applying the divergence theorem, we obtain
\begin{align*}
(\operatorname{curl} \varphi, \operatorname{curl} \psi)
&= \int_\Omega \nabla (u - \tilde{u}) \cdot \operatorname{curl} \psi \, dx \\
&= \int_\Gamma (u - \tilde{u}) (\operatorname{curl} \psi \cdot \mathbf{n}) \, d\Gamma 
   - \int_\Omega (u - \tilde{u}) \operatorname{div}(\operatorname{curl} \psi) \, dx.
\end{align*}
Since $\operatorname{div} \circ \operatorname{curl} = 0$, the volume integral vanishes and
\begin{align*}
(\operatorname{curl} \varphi, \operatorname{curl} \psi)
&= \int_\Gamma (u - \tilde{u}) (\operatorname{curl} \psi \cdot \mathbf{n}) \, d\Gamma
= \langle  g - \tilde{u} , \text{curl} \psi \cdot \mathbf{n} \rangle_{\Gamma}.
\end{align*}
The above deduction shows that $\varphi$ satisfies \eqref{19}. Now we obtain by \eqref{20} that $\nabla u_c=\nabla u$. Then $u_c-u$ is a constant which can be corrected by Step 4 and finally, we are lead to that $u^*=u$. The proof is completed. 
\end{proof}

The four-dimensional construction described above is a special case of a general framework that applies in any dimension \(d \ge 4\). To see this, it is natural to adopt the language of differential forms. On a smooth manifold, \(\Omega^k\) denotes the space of all smooth \(k\)-forms and it is more precisely the space of smooth sections of the \(k\)-th exterior power of the cotangent bundle. On \(\mathbb{R}^d\) endowed with the Euclidean metric, let \(\mathrm{d}: \Omega^k \to \Omega^{k+1}\) denote the exterior derivative and \(\delta: \Omega^k \to \Omega^{k-1}\) the codifferential (the formal adjoint of \(\mathrm{d}\)). These two operators, together with the identities \(\mathrm{d}^2 = 0\) and \(\delta^2 = 0\), form the foundation of de Rham cohomology and Hodge theory. They are precisely the mathematical reason why the framework in the paper can be seamlessly extended from 2D/3D/4D to any higher dimension.

The operators \(\operatorname{curl}\), \(\operatorname{rot}\) and \(\operatorname{div}\) can be naturally extended from the 4D case to higher-dimensional cases. For a vector \(v = (v_1,\dots,v_d)^\top\),
\[
(\operatorname{rot} v)_{ij} = \partial_i v_j - \partial_j v_i, \qquad 1 \le i < j \le d,
\]
which is an antisymmetric second-order tensor (a \(2\)-form). For a \(2\)-form, the curl yields a vector field (a \(1\)-form). The divergence of a vector field is \(\operatorname{div} v = \sum_{i=1}^d \partial_i v_i\). The identities \(\mathrm{d}^2 = 0\) and \(\delta^2 = 0\) immediately imply \(\operatorname{rot} \circ \nabla = 0\) and \(\operatorname{div} \circ \operatorname{curl} = 0\) in every dimension. Because of this geometric structure, the decomposition \(\nabla (u - \tilde{u}) = \operatorname{curl} \varphi\) and the variational arguments used in the proof of \cref{thm:2.3} can be extended to any \(d \ge 4\).

\section{High-dimensional natural deep Ritz method}
\label{sec:3}

In this section, we translate the variational formulations derived in \cref{sec:2} for various dimensions into a set of loss functions that can be directly optimized by neural networks. This establishes the Natural Deep Ritz Method (NatDRM) applicable to any dimension \(d \ge 2\).

As described in \cref{sec:2}, the solution of the essential boundary value problem \eqref{1} is obtained by solving three subproblems with pure Neumann (natural) boundary conditions. Each subproblem is elliptic and can be efficiently handled by the deep Ritz method without any boundary penalty terms. \Cref{tab:lossmap} maps each subproblem to its energy and discrete loss.

\begin{table}[htbp]
\centering
\caption{Subproblems, weak forms, and NatDRM losses.\label{tab:lossmap}}
\small
\begin{tabular}{|c|c|c|c|}
\hline
Step & Continuous & Energy idea & Loss \\
\hline
1 & \eqref{6}, \eqref{12}, \eqref{18} & Neumann Poisson for \(\tilde u\) & \eqref{21} \\
2 & \eqref{7}, \eqref{13}, \eqref{19} & \(\|\mathcal{C}\varphi\|^2\) + boundary & \eqref{22} \\
3 & \eqref{8}, \eqref{14}, \eqref{20} & gradient matching & \eqref{23} \\
\hline
\end{tabular}
\end{table}

Let \(\varphi_{\rm NN}(d,1)\) denote scalar-valued networks with \(d\)-dimensional input; \(\varphi_{\rm NN}(d,k)\) denotes \(k\)-component outputs (vector or \(\binom{d}{2}\) tensor components).

$\mathbf{Subproblem 1: Computation \ of \ \tilde{u}.}$

For the subproblem 1 the corresponding energy functional is
\[
\mathcal{E}_1(\tilde{u}) = \frac12 \int_\Omega |\nabla\tilde{u}|^2 \,{\rm d}x - \int_\Omega f\,\tilde{u} \,{\rm d}x.
\]
We introduce
\[
c_1 = \frac{1}{|\Gamma|}\int_\Gamma \tilde{u}(x)\,{\rm d}\Gamma,
\]
and define the loss function
\begin{equation}\label{21}
\mathcal{L}_1(\tilde{u}) = \sum_{(x_j,\omega_j)\in D} \Big[\frac12 |\nabla\tilde{u}(x_j)|^2 \omega_j - f(x_j)\big(\tilde{u}(x_j)-c_1\big)\omega_j\Big] + c_1^2, 
\end{equation}
where \(D\) and \(D_\Gamma\) are the sets of quadrature points and weights for the domain \(\Omega\) and its boundary \(\Gamma\), respectively. The term \(c_1^2\) is added as a regularization to make the solution unique.

$\mathbf{Subproblem 2: Computation \ of \ the \ potential \ \varphi.}$ 

The way boundary data enter the subproblem 2 depends on the spatial dimension. In all cases the bilinear form is \((\mathcal{C}\varphi,\mathcal{C}\psi)\) with a suitable curl operator \(\mathcal{C}\), and the right-hand side is a boundary integral determined by the mismatch \(g-\tilde{u}|_\Gamma\).\\
$\mathbf{Case \ d = 2.}$ \(\varphi\) is a scalar function. The energy functional corresponding to the weak form \eqref{7} is
\[
\mathcal{E}_2^{\,d=2}(\varphi) = \frac12\int_\Omega |\operatorname{curl}\varphi|^2\,{\rm d}x - \int_\Gamma \partial_t(g-\tilde{u})\,\varphi\,{\rm d}\Gamma.
\]
$\mathbf{Case \ d = 3.}$ \(\varphi\) is a vector field. The energy functional corresponding to the weak form \eqref{13} is
\[
\mathcal{E}_2^{\,d=3}(\varphi) = \frac12\int_\Omega |\operatorname{curl}\varphi|^2\,{\rm d}x - \int_\Gamma \nabla_\Gamma(g-\tilde{u})\cdot\varphi\,{\rm d}\Gamma.
\]
$\mathbf{Case \ d \ge 4.}$ \(\varphi\) is an antisymmetric second-order tensor (2-form). The energy functional corresponding to the weak form \eqref{19} is
\[
\mathcal{E}_2^{\,d\ge4}(\varphi) = \frac12\int_\Omega |\operatorname{curl}\varphi|^2\,{\rm d}x - \int_\Gamma (g-\tilde{u})\,(\operatorname{curl}\varphi\cdot\mathbf{n})\,{\rm d}\Gamma.
\]
The solution of each variational problem is unique only up to a kernel. This kernel can be fixed by adding a regularization term that penalizes a suitable mean of \(\varphi\) on the boundary, which appears in the loss function as \(\overline{\varphi}^{\,2}\). For 2D we use \((\int_\Gamma \varphi)^2\); for 3D and higher dimensions we penalize the square of the mean of the normal component, e.g., \((\int_\Gamma \varphi\cdot\mathbf{n})^2\).

Hence we define the second loss function in a unified manner as
\begin{equation}\label{22}
\mathcal{L}_2(\varphi) = \sum_{(x_j,\omega_j)\in D} \frac12 |\mathcal{C}\varphi(x_j)|^2\omega_j \;+\; \mathcal{B}_2(\varphi;\tilde{u},g) \;+\; \overline{\varphi}^{\,2}, 
\end{equation}
where \(\mathcal{C}\varphi = \operatorname{curl}\varphi\) for \(d\ge2\), and the boundary term \(\mathcal{B}_2\) is the discretized form of the corresponding boundary integral, namely
\[
\mathcal{B}_2^{\,2D} = -\sum_{(x_j,\omega_j)\in D_\Gamma} \partial_t(g-\tilde{u})(x_j)\,\varphi(x_j)\,\omega_j,
\]
\[
\mathcal{B}_2^{\,3D} = -\sum_{(x_j,\omega_j)\in D_\Gamma} \nabla_\Gamma(g-\tilde{u})(x_j)\cdot\varphi(x_j)\,\omega_j,
\]
\[
\mathcal{B}_2^{\,d\ge4} = -\sum_{(x_j,\omega_j)\in D_\Gamma} (g(x_j)-\tilde{u}(x_j))\,\bigl(\operatorname{curl}\varphi(x_j)\cdot\mathbf{n}(x_j)\bigr)\,\omega_j.
\]
For 2D the squared mean \(\overline{\varphi}^{\,2}\) can be taken as \((\int_\Gamma \varphi)^2\); for 3D and higher dimensions we penalize the square of the mean of the normal component, e.g., \((\int_\Gamma \varphi\cdot\mathbf{n})^2\).

$\mathbf{Subproblem \ 3: Recovery \ of \ u_c.}$ 

The subproblem~3 reconstructs the solution up to a constant. With \(\mathcal{C}\varphi=\operatorname{curl}\varphi\), 
\[
(\nabla u_c,\nabla v) = (\nabla\tilde{u} -\,\mathcal{C}\varphi,\nabla v).
\]
The corresponding least-squares loss is
\begin{equation}\label{23}
\mathcal{L}_3(u_c) = \sum_{(x_j,\omega_j)\in D} \bigl|\nabla u_c(x_j) - \nabla\tilde{u}(x_j) +\,\mathcal{C}\varphi(x_j)\bigr|^2\omega_j \;+\; \Bigl(\sum_{(x_j,\omega_j)\in D_\Gamma} (u_c(x_j)-g(x_j))\omega_j\Bigr)^2 .
\end{equation}
The second term fixes the additive constant by matching boundary means. For \(d\ge3\), \(\mathcal{C}\varphi\) is vector-valued and matching is component-wise.

\begin{algorithm}[htbp]
\caption{NatDRM training (joint mode)\label{alg:natdrm}}
\begin{algorithmic}
\STATE Sample quadrature sets \(D\subset\Omega\), \(D_\Gamma\subset\Gamma\); choose networks \(\tilde u_\theta\), \(\varphi_\eta\), \(u_{c,\zeta}\).
\STATE \textbf{repeat} until convergence
\STATE \quad Minimize \(\mathcal{L}_1(\tilde u_\theta)\) using \eqref{21} (gauge \(c_1\)).
\STATE \quad Minimize \(\mathcal{L}_2(\varphi_\eta)\) using \eqref{22} with \(\tilde u_\theta|_\Gamma\) and \(g\).
\STATE \quad Minimize \(\mathcal{L}_3(u_{c,\zeta})\) using \eqref{23} with \(\tilde u_\theta\), \(\varphi_\eta\).
\STATE \textbf{end repeat}
\STATE Return \(u \approx u_{c,\zeta} - C\) with \(C\) from Step~4 in \cref{sec:2}.
\end{algorithmic}
\end{algorithm}

In practice we minimize \(\mathcal{L}=\mathcal{L}_1+\mathcal{L}_2+\mathcal{L}_3\) simultaneously (\cref{alg:natdrm}, joint mode). An alternative is to alternate the optimization of the three subproblems, where each subproblem’s network is updated while the other two networks are frozen. This is achieved by gradient blocking: for the loss \(\mathcal{L}_1\), only the parameters of \(\tilde{u}_\theta\) receive gradients; for \(\mathcal{L}_2\), the parameters of \(\varphi_\eta\) receive gradients while the outputs of \(\tilde{u}_\theta\) and \(u_{c,\zeta}\) are treated as constants; similarly, for \(\mathcal{L}_3\), only the parameters of \(u_{c,\zeta}\) are updated, and the contributions from \(\tilde{u}_\theta\) and \(\varphi_\eta\) are detached. This alternating scheme completely decouples the gradient flows among the three subnetworks, preventing interference and often leading to more stable convergence. In practice, one can perform one Adam step for each subproblem in a round‑robin fashion within each epoch.

In our experiments, the alternating optimization scheme yields larger relative \(L^2\) errors than joint training, yet requires less runtime per epoch. This behavior can be explained as follows. Joint optimization simultaneously updates all three subnetworks by minimizing the total loss \(\mathcal{L}=\mathcal{L}_1+\mathcal{L}_2+\mathcal{L}_3\); the gradients are computed from a single consistent objective, which allows the three networks to coordinate and better approach the global optimum. In contrast, alternating optimization solves each subproblem with the other two networks held fixed, creating a “lagged” coupling: the boundary term \(v = \nabla u - \nabla u_1\) used when updating \(\varphi\) becomes outdated after \(u\) or \(u_1\) changes. This inconsistency traps the iterates in suboptimal local minima, hence the higher error. The shorter runtime, however, arises because each backward pass in alternating mode involves only one subnetwork’s computation graph, which is smaller and less memory‑intensive; the joint backward pass must accommodate all three graphs simultaneously, incurring extra synchronization and memory overhead. A quantitative comparison of the two strategies, including relative \(L^2\) errors and time across multiple runs, is presented later in \Cref{sec:4} (see \Cref{tab:joint_vs_alt}).

\subsection{Extensions to variable coefficients and semilinear terms }
\label{sec:3-extensions}

For the variable-coefficient problem \eqref{26}, subproblem~1 becomes \(-\mathrm{div}(\mathcal{A}^2\nabla\tilde u)=f\) with the same Neumann compatibility on \(\Gamma\). The Ritz loss replaces \eqref{21} by
\[
\mathcal{L}_1^{\mathrm{vc}}(\tilde u) = \sum_{(x_j,\omega_j)\in D}\Big[\tfrac12\,\mathcal{A}(x_j)\nabla\tilde u(x_j)\cdot \mathcal{A}(x_j)\nabla\tilde u(x_j)\,\omega_j - f(x_j)\big(\tilde u(x_j)-c_1\big)\omega_j\Big]+c_1^2,
\]
while $L_2$, $L_3$ retain the weighted Poisson-type gradient decomposition on $g-\tilde{u}$. This formulation is a natural high-dimensional extension of the 2D variable-coefficient result rigorously proven in \cite{15}, where the decomposition is justified by the weighted de Rham complex theory. In our numerical experiments, we use a diagonal positive-definite coefficient matrix $\mathcal{A}(x)$, which satisfies the conditions for the weighted Hodge decomposition. The mathematical justification for general non-diagonal variable coefficients remains an open problem for future investigation.

For general semilinear elliptic problems of the form $-\Delta u = f(x, u)$ with Dirichlet boundary conditions, we employ a Picard fixed-point iteration as the outer loop to handle the nonlinearity. At the $k$-th outer iteration:

1.  Evaluate the nonlinear right-hand side $f^{(k)}(x) = f\left(x, u^{(k)}(x)\right)$ using the current approximate solution $u^{(k)} \approx u_{c}^{(k)} - C^{(k)}$ obtained from the previous iteration.

2.  Update the intermediate functions $\tilde{u}$, $\varphi$, and $u_c$ by minimizing the joint loss $\mathcal{L} = \mathcal{L}_1 + \mathcal{L}_2 + \mathcal{L}_3$, where $\mathcal{L}_1$ uses the updated $f^{(k)}(x)$ as the right-hand side, while $\mathcal{L}_2$ and $\mathcal{L}_3$ remain unchanged as they only depend on the boundary mismatch $g - \tilde{u}|_{\Gamma}$ and the gradient matching condition.

The Picard iteration is chosen for its computational efficiency and numerical robustness: it avoids the expensive Hessian computation required by Newton-type methods, which is particularly prohibitive in high-dimensional neural network settings. It also exhibits better stability with the small-weight initialization of neural networks. For strongly nonlinear problems, a relaxation parameter $\alpha \in (0, 1]$ can be introduced to update the solution as $u^{(k+1)} = \alpha u^{(k+1)}_{\text{new}} + (1-\alpha) u^{(k)}$ to ensure convergence.

In numerical experiments, we adopted the following convergence criteria:

Inner training convergence: For each outer iteration, the joint loss $\mathcal{L}$ is minimized until its relative reduction per epoch is less than $\epsilon_{\text{inner}} = 10^{-6}$, or a maximum of 300 Adam epochs is reached.

Outer iteration convergence: The outer loop terminates when the relative $L^2$ error between consecutive iterates satisfies
  $$\frac{\|u^{(k+1)} - u^{(k)}\|_{L^2(\Omega)}}{\|u^{(k)}\|_{L^2(\Omega)}} < \epsilon_{\text{outer}} = 10^{-4},$$
or a maximum of 15 outer iterations is performed to prevent excessive computation.

\subsection{Training implementation}

Losses \eqref{21}-\eqref{23} are evaluated by automatic differentiation (AD) in PyTorch 2.2 \cite{20}. For \(d \leq 3\), \(\tilde{u}\), \(\varphi\), and \(u_c\) are scalar- or vector-valued fully connected residual networks; for \(d\ge4\), \(\varphi\) outputs \(\binom{d}{2}\) independent antisymmetric components. All derivatives (gradient, curl, surface gradient) are obtained via AD. Numerical results are reported in \cref{sec:4}.

In numerical experiments, we need to further consider Gauge-fixing regularization. The curl operator has a non-trivial kernel due to the Poincaré lemma on simply connected domains:

In 2D, \(\text{curl}\,\varphi = 0\) if and only if \(\varphi\) is a constant;

In 3D, \(\text{curl}\,\varphi = 0\) if and only if \(\varphi = \nabla q\) for some scalar function \(q\);

In \(d \geq 4\), \(\text{curl}\,\varphi = 0\) if and only if \(\varphi = \text{rot}\,\psi\) for some vector field \(\psi\) (1-form).

Consequently, the variational formulation of Subproblem 2 admits solutions that are unique only up to an element of this kernel. Crucially, since the final solution depends only on \(\text{curl}\,\varphi\) (not \(\varphi\) itself) in the gradient matching step (Subproblem 3), this non-uniqueness does not affect the correctness of the method. However, appropriate regularization is required to prevent the norm of \(\varphi\) from growing unbounded during training and to ensure numerical stability.

To this end, we add simple boundary-based gauge-fixing terms to the loss function \eqref{22}:

For \(d=2\), we penalize the squared boundary mean \(\left(\int_{\Gamma} \varphi \, d\Gamma\right)^2\), which uniquely fixes the constant kernel and guarantees a unique solution;

For \(d \geq 3\), we penalize the squared mean of the normal component on the boundary \(\left(\int_{\Gamma} \varphi \cdot n \, d\Gamma\right)^2\).

Mathematically, the single scalar constraint used for \(d \geq 3\) is insufficient to eliminate the entire infinite-dimensional kernel. Nevertheless, numerical experiments demonstrate that this regularization, combined with implicit regularization effects inherent in deep learning (including small-weight initialization, adaptive optimization dynamics, and limited network expressivity), is sufficient to stabilize training. Additionally, all experiments use a small weight decay (\(\text{weight\_decay}=10^{-5}\)) in the Adam optimizer, which further penalizes large values of \(\varphi\) and prevents drift in the kernel directions.

For applications requiring mathematically rigorous uniqueness, a Hodge gauge-fixing strategy based on the codifferential operator \(\delta\) can be employed. This involves adding the term \(\alpha (\delta \varphi, \delta \varphi)\) to the energy functional, where \(\alpha > 0\) is a fixed regularization parameter (typically \(\alpha=1\)). This modified bilinear form is coercive by the Hodge decomposition theorem, ensuring a unique solution while preserving the value of \(\text{curl}\,\varphi\) required for Subproblem 3.




\section{Numerical experiments}
\label{sec:4}

In this section, we apply the proposed Natural Deep Ritz Method to solve Poisson equations in three, four, and six dimensions, and conduct comprehensive comparisons with the standard DRM and PINN.

\subsection{Experiment setup}

All numerical experiments are implemented using PyTorch 2.2 \cite{20}. In the 3D experiment, the network architecture employs a ResNet with 5 residual blocks (ResBlock); whereas for problems of dimension greater than three, the number of residual blocks is reduced to 2. This reduction is adopted to enhance computational efficiency and stabilize optimization performance when tackling high‑dimensional problems. Each ResBlock consists of two activation layers and two linear transformations. To ensure a fair comparison of parameter counts across methods, the hidden widths of NatDRM, DRM, and PINN are set to 20, 35, and 35, respectively. Considering that NatDRM simultaneously optimizes the loss functions of three subproblems and needs to represent a tensor-valued potential \(\varphi\) (which requires outputting \(\binom{d}{2}\) independent components for \(d\geq 4\)), its actual number of parameters remains on the same order as the other two methods.

The tested activation functions include: the rectified quadratic unit with Lipschitz regularization (ReQUr), the rectified cubic unit with Lipschitz regularization (ReCUr), and the hyperbolic tangent (Tanh). They are defined respectively as:
\begin{align*} &\operatorname{ReQUr}(x) = \operatorname{ReLU}^2(x) - \operatorname{ReLU}^2(x - 0.5), \\ &\operatorname{ReCUr}(x) = \operatorname{ReLU}^3(x) - 2\operatorname{ReLU}^3(x - 0.5) + \operatorname{ReLU}^3(x - 1). \end{align*} 
Additional details concerning ResNet and rectified power activations appear in \cite{he2016deep,li2020better,yu2021onsagernet}. Training data are generated using 5th-order composite Gaussian quadrature. For three- and four-dimensional problems, 20,000 quadrature points with corresponding weights are sampled both in the interior and on the boundary. An additional 10,000 uniformly distributed test points are generated to evaluate generalization error. For six-dimensional problems, 40,000 quadrature points are used for interior and boundary training, and 20,000 uniform test points are generated. Both the Adam and L-BFGS optimizers are employed. For Adam, we train for 100 epochs with batch sizes of 200, 300, and 500 for 3D, 4D, and 6D problems respectively. We set the initial learning rate to 0.005 and employ a learning rate scheduler, CosineAnnealingLR \cite{21}. For L-BFGS, we perform 50 steps with a history size of 100 and 60 inner iterations per step. All derivatives (gradient, curl, tangential gradient, and Laplacian) are evaluated via automatic differentiation.

\subsection{Benchmark examples}

To systematically verify the performance of NatDRM under different dimensions and smoothness conditions, we conduct numerical experiments on the following four benchmark examples.

Example 1:
\begin{align}\label{24}
\left\{
\begin{aligned} -\Delta u &= d\sin(\sum_{i=1}^{d}x_i) - 2d, && \text{in } \Omega = [-1,1]^d, \\ u &= \sum_{i=1}^{d}x_i^2 + \sin(\sum_{i=1}^{d}x_i), && \text{on } \partial\Omega . \end{aligned}
\right.
\end{align}
This example is globally smooth, making it an ideal benchmark for testing the convergence and accuracy of numerical methods.

Example 2:
\begin{align}\label{25}
\left\{
\begin{aligned} -\Delta u &= \dfrac{\pi^2}{4}\sum_{i=1}^{d}\sin\!\Big(\dfrac{\pi}{2}x_i\Big), && \text{in } \Omega = [0,1]^d, \\ u &= \sum_{i=1}^{d}\sin\!\Big(\dfrac{\pi}{2}x_i\Big), && \text{on } \partial\Omega . \end{aligned}
\right.
\end{align}
The solution of this example is a superposition of multiple one-dimensional sine functions, with a typical high-dimensional separable structure, which is commonly used to test the performance of high-dimensional PDE solvers.

Example 3:

We further consider a second order problem of divergence form, namely an elliptic equation with smooth variable coefficients and a Dirichlet boundary condition:
\begin{align}\label{26}
\left\{
\begin{aligned}
-\text{div}(\mathcal{A}^2 \nabla u) -f&= 0, && \text{in } [-1,1]^d, \\
u -g&= 0, && \text{on } \partial\Omega.
\end{aligned}
\right.
\end{align}
The variable coefficient matrix is taken as $ \mathcal{A}(x)=\mathrm{diag}\big(1+x_1^2,\;1,\;1,\dots,1\big) \in\mathbb{R}^{d\times d} $, and the exact solution is given as $u(x)=e^{\cos\big(x_1+x_2^2\big)} $.

Example 4:

To further test a nonlinear equation of greater complexity, the following semilinear Poisson equation with Dirichlet boundary condition is considered as a test case.
\begin{align}\label{27}
\left\{
\begin{aligned} -\Delta u &= 2du + 4u\ln{u}, && \text{in } \Omega = [-1,1]^d, \\ u &= e^{-\sum_{i=1}^{d}x_i^2}, && \text{on } \partial\Omega . \end{aligned}
\right.
\end{align}

\subsection{Experimental results}

\subsubsection{Efficiency and Accuracy Comparison}

Table \ref{tab:1} summarizes the CPU time (in seconds) and relative \(L^2\) errors of DRM, PINN, and NatDRM across three dimensions. For DRM and PINN, multiple penalty parameters \(\beta \in \{10, 100, 1000\}\) are tested, whereas NatDRM requires only a single run without any parameter tuning.

\begin{table}[htbp]
\centering
\caption{CPU time (s) and relative \(L^2\) errors of different methods on 3D, 4D, and 6D examples.}
\label{tab:1}
\small
\begin{tabular}{|c|c|c|c|c|c|c|c|c|c|}
\hline
Dim & Ex & Met & \multicolumn{3}{c|}{DRM} & \multicolumn{3}{c|}{PINN} & NatDRM \\
\cline{4-9}
    &         &        & $\beta=10$ & 100 & 1000 & $\beta=10$ & 100 & 1000 & \\
\hline
3D  & 1 & time  & 198.2 & 196.1 & 196.9 & 401.8 & 402.9 & 401.3 & 762.3 \\
    &   & error & 7.2(-2) & 7.8(-3) & 1.2(-2) & 1.1(-2) & 7.5(-3) & 8.2(-3) & 4.5(-3) \\
    & 2 & time  & 195.7 & 194.7 & 192.9 & 452.3 & 451.9 & 451.8 & 773.2 \\
    &   & error & 4.6(-3) & 2.9(-3) & 9.8(-3) & 3.5(-3) & 2.8(-3) & 2.6(-3) & 1.5(-3) \\
4D  & 1 & time  & 106.2 & 112.1 & 101.9 & 309.6 & 321.8 & 308.9 & 465.6 \\
    &   & error & 4.6(-1) & 8.4(-2) & 5.2(-3) & 1.5(-1) & 1.0(-2) & 2.9(-3) & 3.1(-3) \\
    & 2 & time  & 103.7 & 115.7 & 102.9 & 308.2 & 309.3 & 307.9 & 576.1 \\
    &   & error & 1.8(-1) & 2.0(-2) & 2.2(-3) & 1.5(-2) & 8.4(-3) & 2.9(-3) & 3.0(-3) \\
6D  & 1 & time  & - & - & 1184.3 & 182.5 & 195.8 & 222.7 & 341.7 \\
    &   & error & $> 1$ & $> 1$ & 1.4(-2) & 3.0(-2) & 4.3(-3) & 2.1(-3) & 3.2(-2) \\
    & 2 & time  & - & 838.2 & 985.8 & 293.9 & 331.6 & 315.8 & 368.0 \\
    &   & error & $> 1$ & 3.4(-2) & 2.0(-3) & 1.9(-2) & 1.3(-3) & 1.5(-3) & 7.7(-3) \\
\hline
\end{tabular}
\end{table}

From \Cref{tab:1}, the following key observations can be made:

(1) $\mathbf{Accuracy \ comparison.}$ In 3D and 4D, NatDRM achieves accuracy comparable to or even better than the optimally tuned DRM or PINN without any parameter tuning. For both examples in these dimensions, the errors of the three methods under their respective optimal parameters lie in the same order of magnitude, exhibiting satisfactory solution accuracy. In the 6D examples, while the best-tuned PINN yields slightly smaller errors, NatDRM remains stable and completely tuning-free, and significantly outperforms DRM. For most $\beta$ values, DRM fails to produce a reasonable numerical solution, and its relative $L^2$ error exceeds 1, which is marked as “$>1$” in the table. Correspondingly, the “-” retained in the time column indicates that the runtime is meaningless for these cases, as DRM cannot converge to a valid solution within a practical computational time frame under these parameter settings.

(2) $\mathbf{Sensitivity \ to \ the \ penalty \ parameter.}$ The accuracy of DRM and PINN strongly depends on the choice of \(\beta\), and the optimal \(\beta\) varies across examples. Taking 3D as an example, the DRM error on \eqref{25}-3D drops from \(4.6\times 10^{-3}\) at \(\beta=10\) to \(2.9\times 10^{-3}\) at \(\beta=100\), then rises to \(9.8\times 10^{-3}\) at \(\beta=1000\). This non‑monotonic behavior indicates that an excessively large \(\beta\) introduces severe numerical stiffness, while an overly small \(\beta\) provides insufficient boundary constraint. Although PINN is generally less sensitive to \(\beta\), its optimal value still differs among cases. In practice, identifying the optimal \(\beta\) requires multiple trial runs for each new problem, substantially increasing the total computational cost.

(3) $\mathbf{Computational \ time \ analysis.}$ In terms of single‑run time, NatDRM generally takes longer than a single run of DRM or PINN. For example, in 4D, NatDRM takes about 466 s (Example 1) and 576 s (Example 2), while DRM takes about 102–112 s per run. The extra cost arises because NatDRM (i) requires forward and backward passes for three loss functions, and (ii) must represent an antisymmetric tensor potential \(\varphi\) in high dimensions (6 independent components in 4D and 15 in 6D), increasing the network output dimension. Crucially, however, DRM and PINN usually require 3–5 (or more) trial runs to determine a suitable \(\beta\), making their total time potentially far exceed a single run of NatDRM. In high dimensions, the time disadvantage of NatDRM becomes relatively smaller as the model size grows, and explicit parallelization of the subproblems can further improve its efficiency.

(4) $\mathbf{Degradation \ in \ high \ dimensions.}$ As the dimension increases from 3D to 6D, the accuracy of all three methods degrades to some extent. For the 6D Example 1 with \(\beta=10\) and \(\beta=100\), DRM even fails to attain meaningful accuracy, indicating that the penalty approach severely breaks down in very high dimensions due to sparse boundary sampling and the stiffness of the penalty term. PINN remains trainable in 6D, but the tuning cost increases significantly. In this regime, the advantage of NatDRM becomes even more pronounced, as it exhibits consistently stable convergence without the risk of tuning failure.

To further investigate the practical trade‑off between joint and alternating training, we compare the two strategies on the 4D Poisson benchmark \eqref{24} under identical settings (network architecture, quadrature points, optimizers, and number of epochs). The alternating scheme follows the gradient‑blocking procedure, where each subproblem is optimized separately while the other two networks are frozen. Five independent runs with different random seeds are performed, and the numerical ranges of the relative \(L^2\) error and the training loss are reported in \Cref{tab:joint_vs_alt}.

\begin{table}[h]
\centering
\caption{Comparison of joint vs. alternating optimization on the 4D Poisson problem \eqref{24}. Results are averaged over five runs with different random seeds.}
\label{tab:joint_vs_alt}
\begin{tabular}{|c|c|c|c|}
\hline
Strategy & Rerror (inner) & error (boundary) & time \\
\hline
Joint & 5.3(±0.5)e-3 & 3.1(±0.1)e-3 & 465(±41) \\
Alternating & 6.8(±0.7)e-2 & 7.6(±0.2)e-2 & 367(±24) \\
\hline
\end{tabular}
\end{table}

As shown in \Cref{tab:joint_vs_alt}, joint training achieves lower relative \(L^2\) errors and a narrower spread, indicating better convergence to the exact solution. Alternating training, while faster per epoch, suffers from larger errors due to the “lagged” coupling among subproblems. This observation confirms that, for the strongly coupled NatDRM loss, the simultaneous update of all three networks is essential for accuracy, whereas the alternating mode is more suitable when computational speed is prioritized over precision.

\subsubsection{Training Behavior and Influence of Activation Functions}

\Cref{fig1}-\Cref{fig3} show the relative \(L^2\) interior and boundary errors versus training epochs for different activation functions. The blue line represents the ReCUr activation function, the orange line represents the ReQUr activation function, and the green line represents the Tanh activation function.

\begin{figure}[tbhp]
\centering
\subfloat[NatDRM: inner error]{\label{fig:ex3NDBerrori}\includegraphics[width=0.49\textwidth]{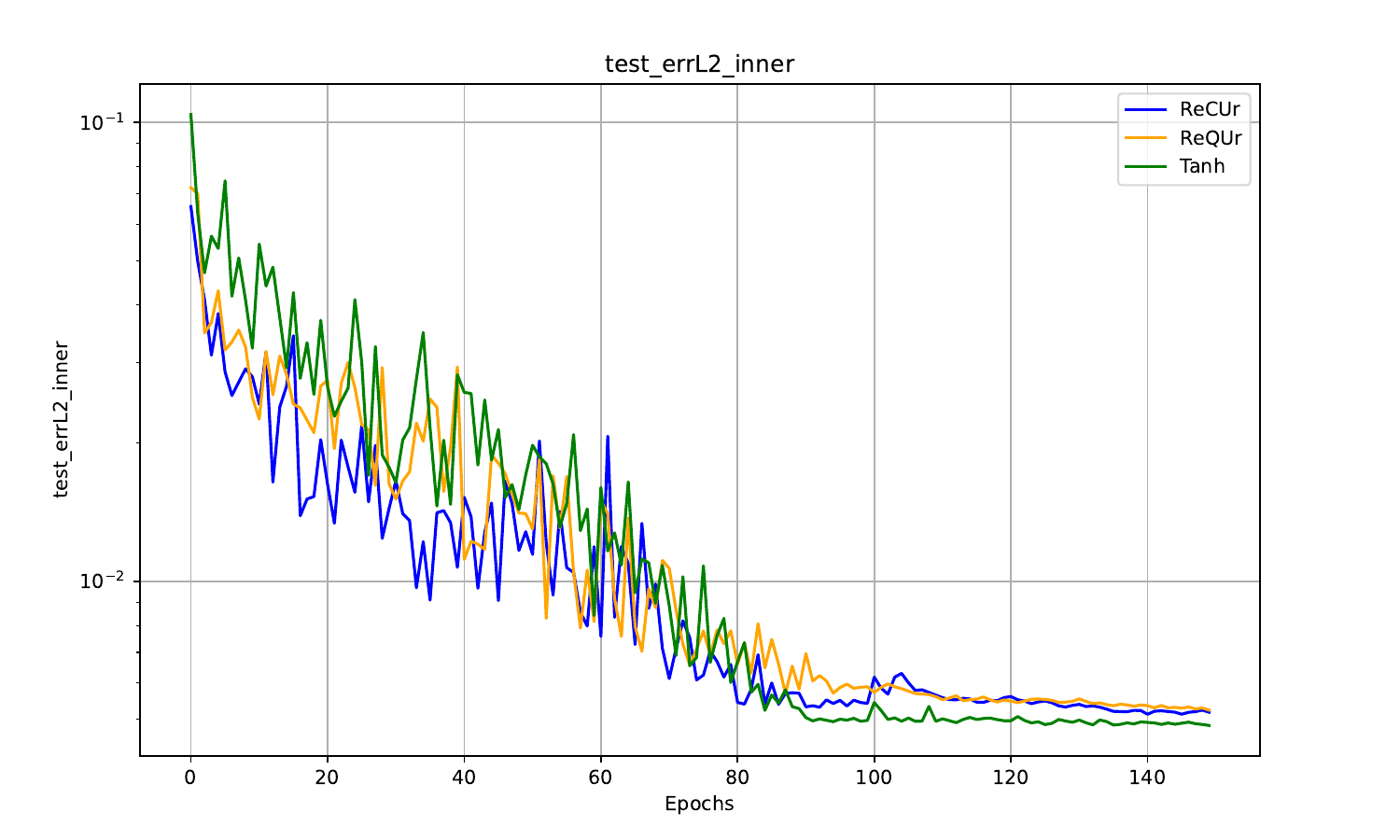}}
\subfloat[NatDRM: boundary error]{\label{fig:ex3NDBerrorb}\includegraphics[width=0.49\textwidth]{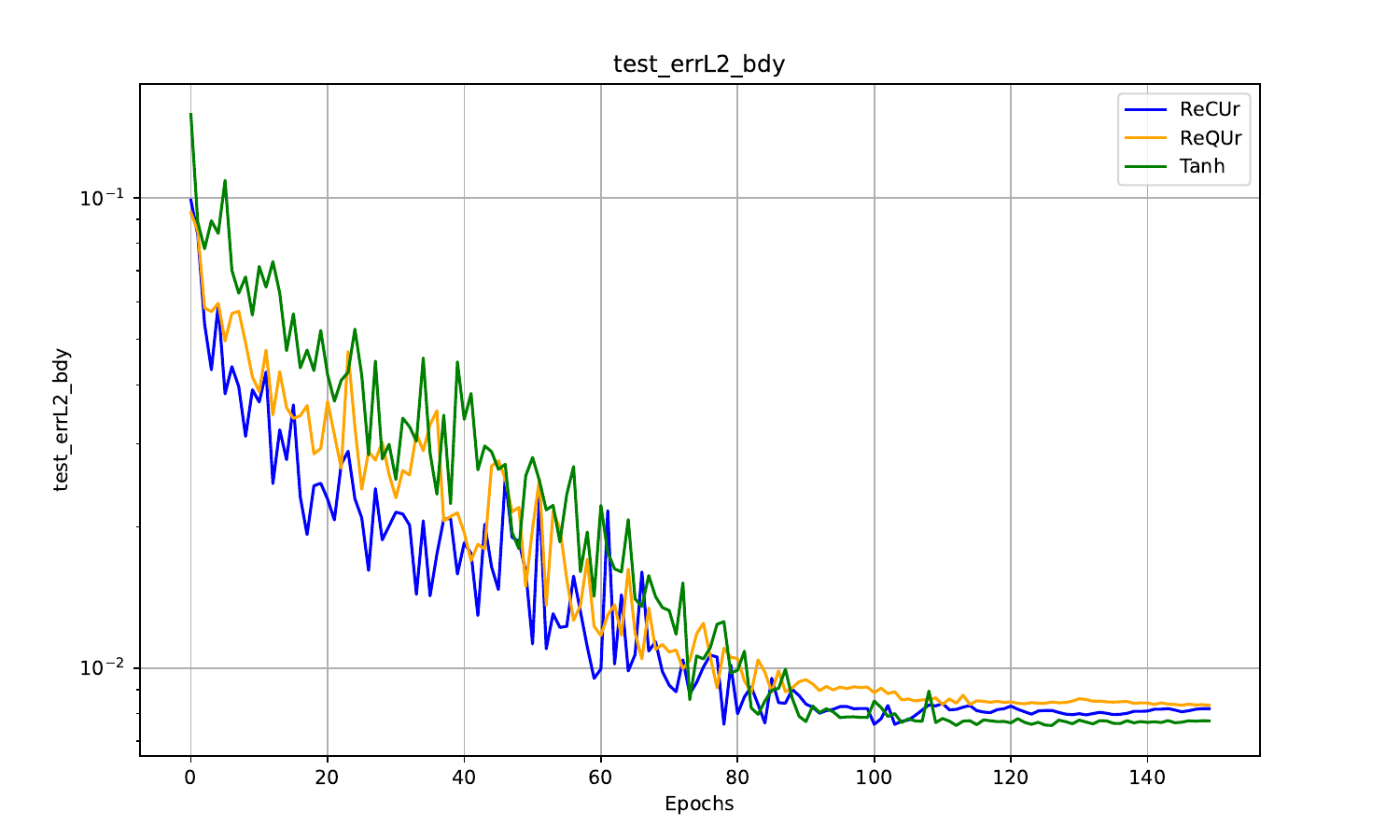}}
 \\
\subfloat[DRM: $beta=100$ inner error]{\label{fig:ex3DRMerrori}\includegraphics[width=0.49\textwidth]{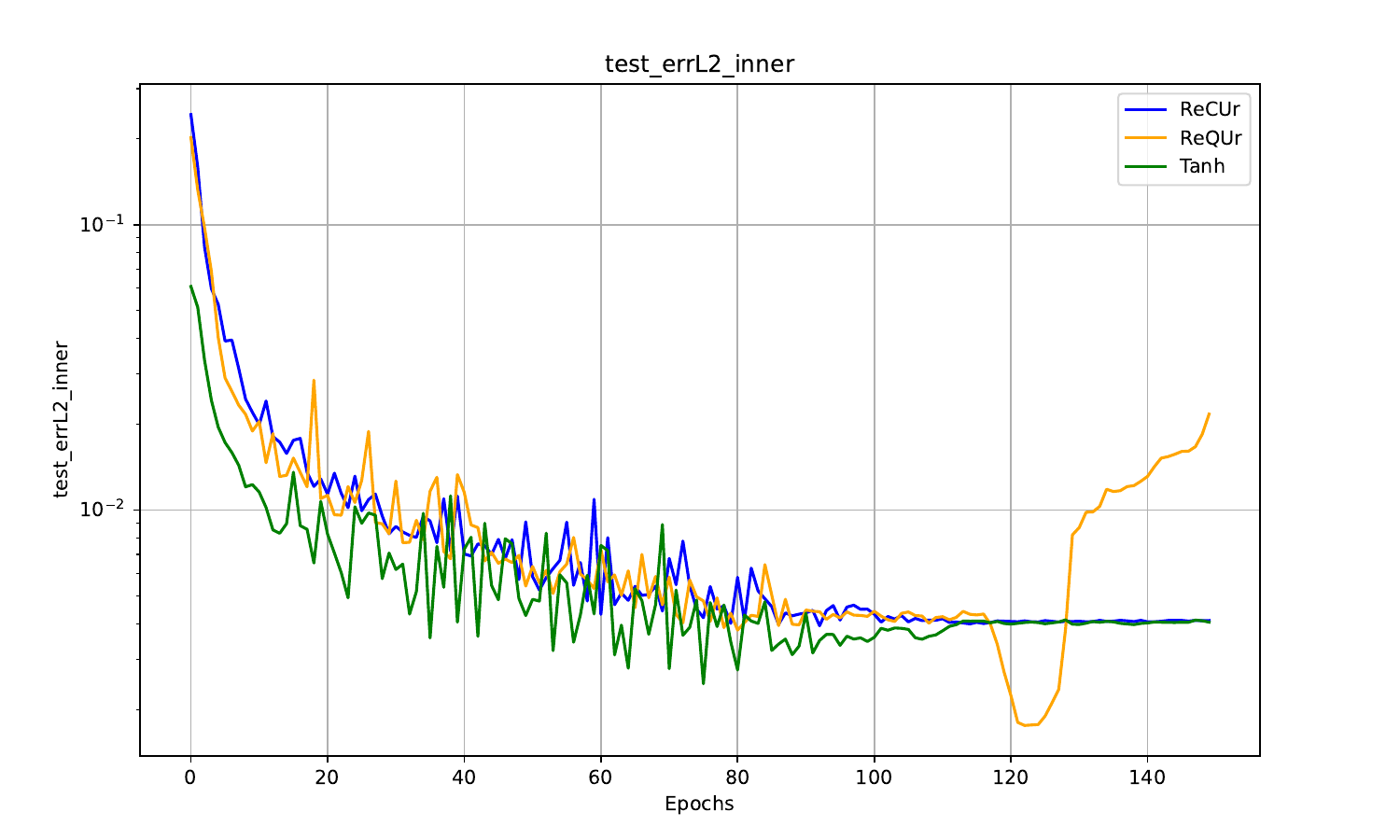}}
\subfloat[DRM: $beta=100$ boundary error]{\label{fig:ex3DRMerrorb}\includegraphics[width=0.49\textwidth]{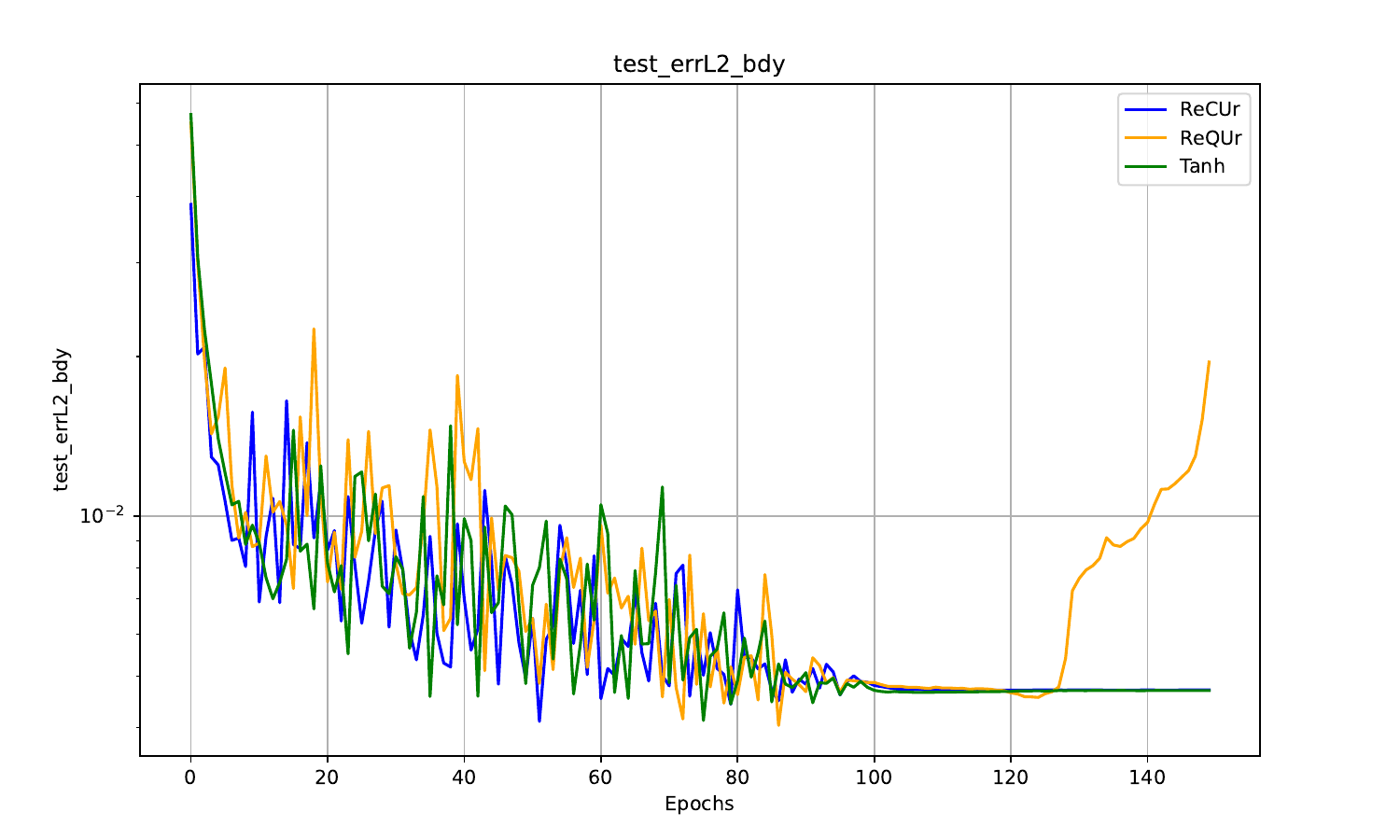}}
\caption{Relative $L^2$ interior and boundary errors for NatDRM and DRM ($\beta=100$) during training on the 3D benchmark \eqref{26}.}
\label{fig1}
\end{figure}

\begin{figure}[tbhp]
\centering
\subfloat[NatDRM: inner error]{\label{fig:ex1NDBerrori}\includegraphics[width=0.49\textwidth]{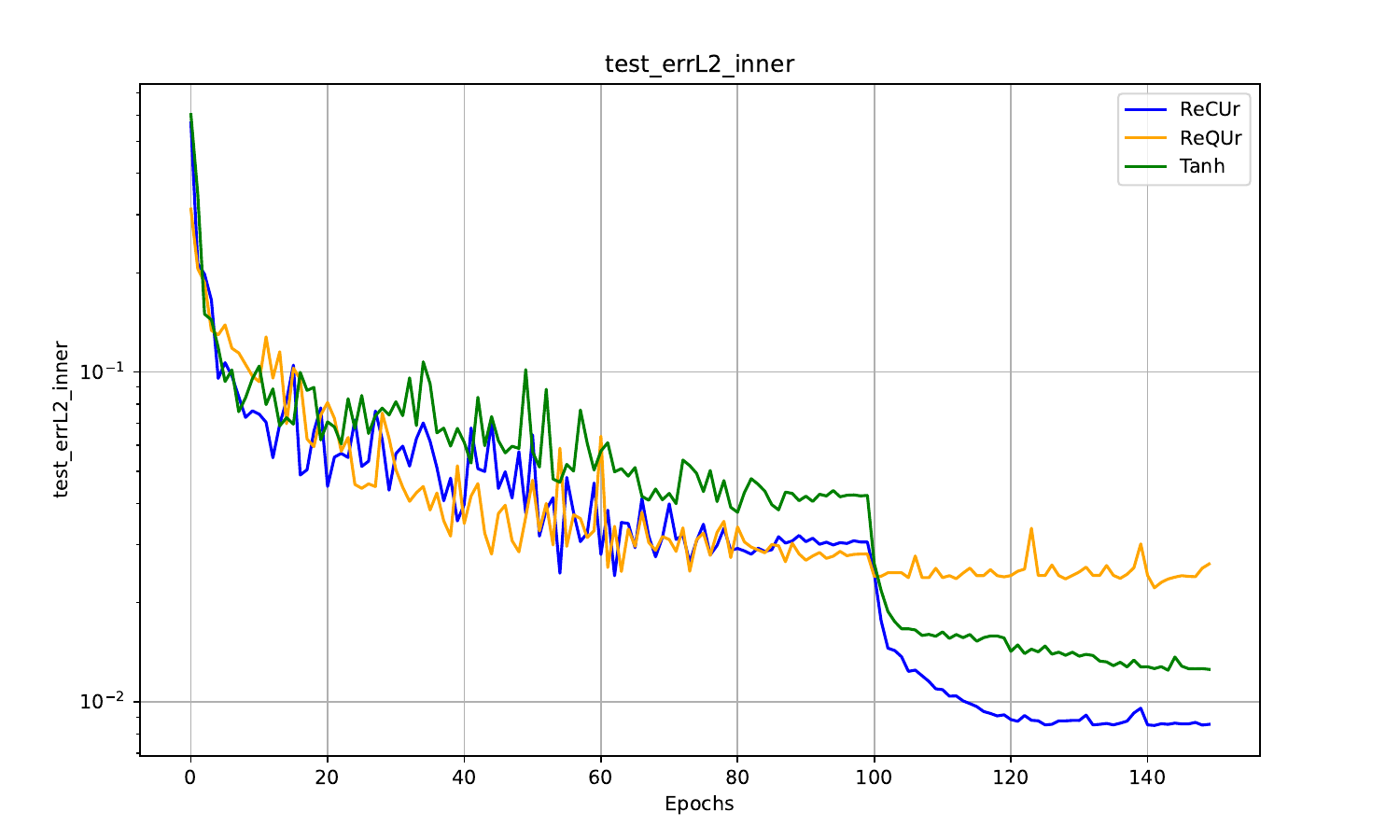}}
\subfloat[NatDRM: boundary error]{\label{fig:ex1NDBerrorb}\includegraphics[width=0.49\textwidth]{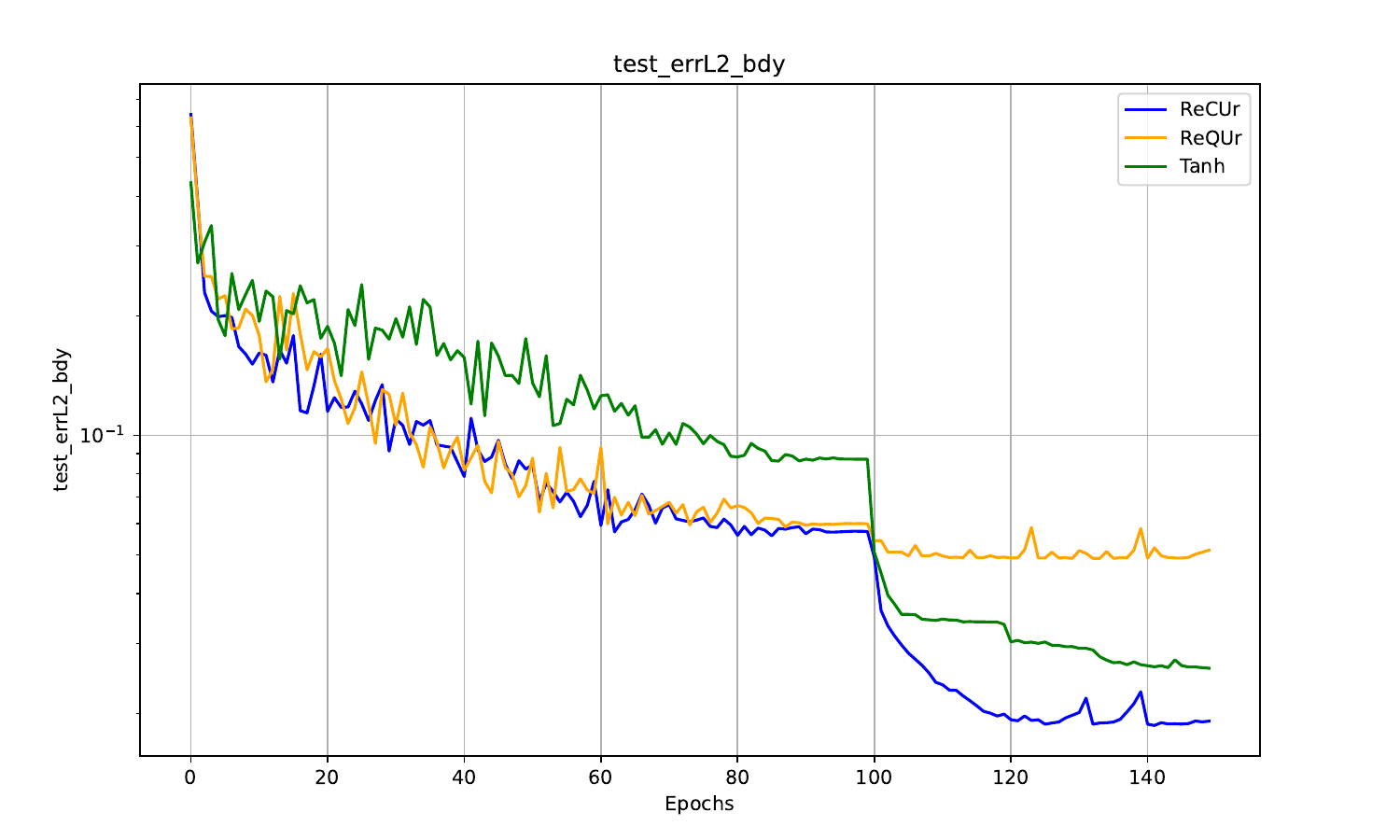}}
\\
\subfloat[DRM: $beta=100$ inner error]{\label{fig:ex1DRMerrori}\includegraphics[width=0.49\textwidth]{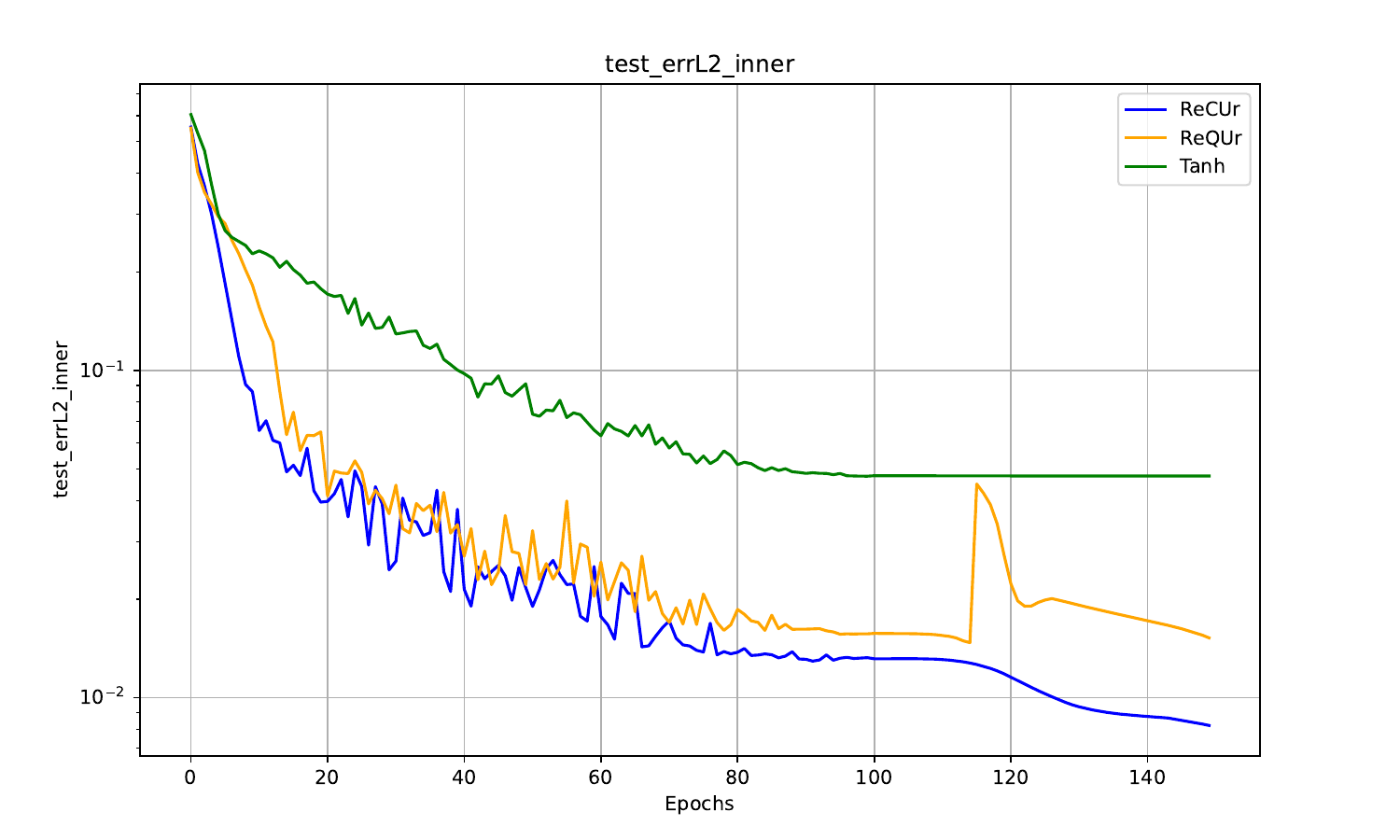}}
\subfloat[DRM: $beta=100$ boundary error]{\label{fig:ex1DRMerrorb}\includegraphics[width=0.49\textwidth]{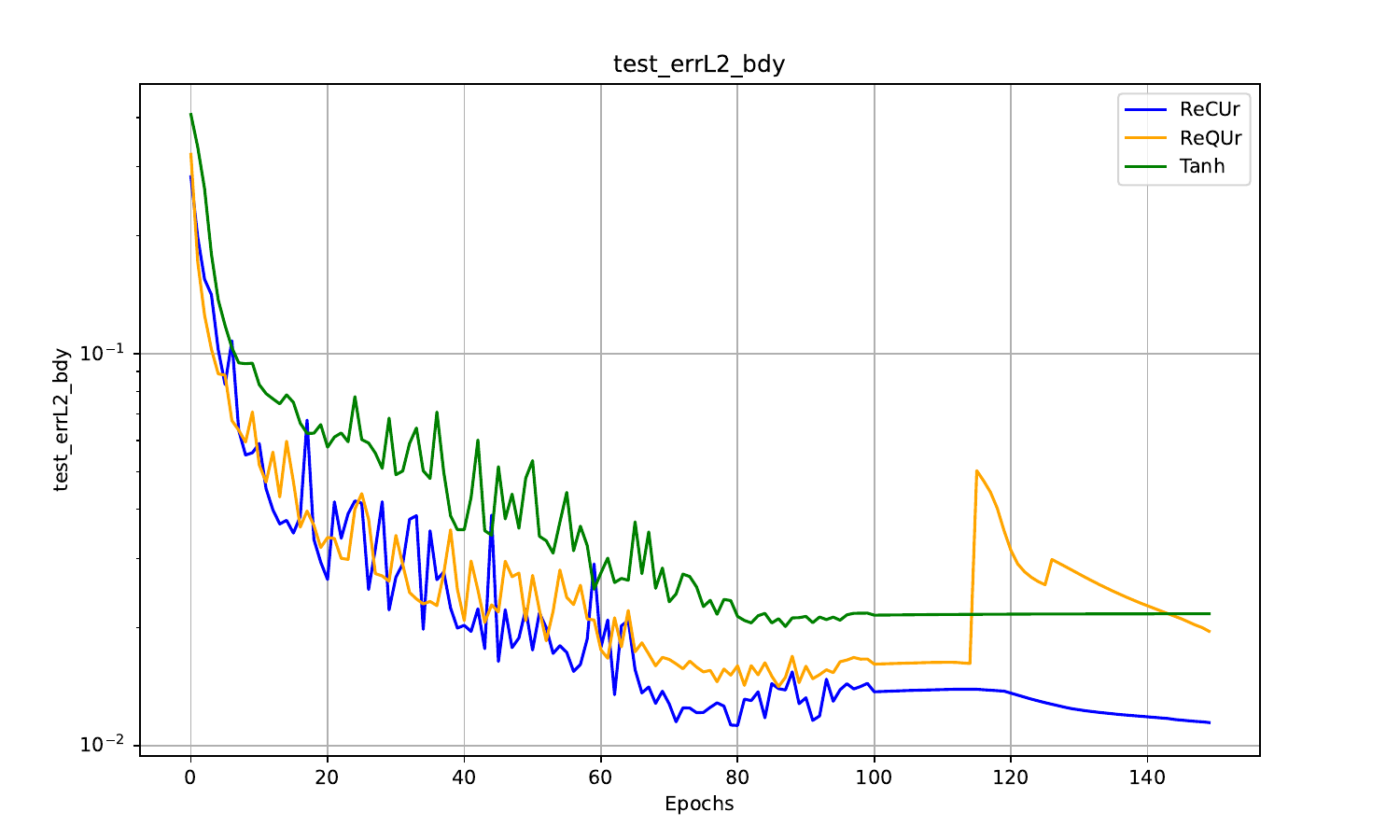}}
\caption{Relative $L^2$ interior and boundary errors for NatDRM and DRM ($\beta=100$) during training on the 4D benchmark \eqref{24}.}
\label{fig2}
\end{figure}

\begin{figure}[tbhp]
\centering
\subfloat[NatDRM: inner error]{\label{fig:ex2NDBerrori}\includegraphics[width=0.49\textwidth]{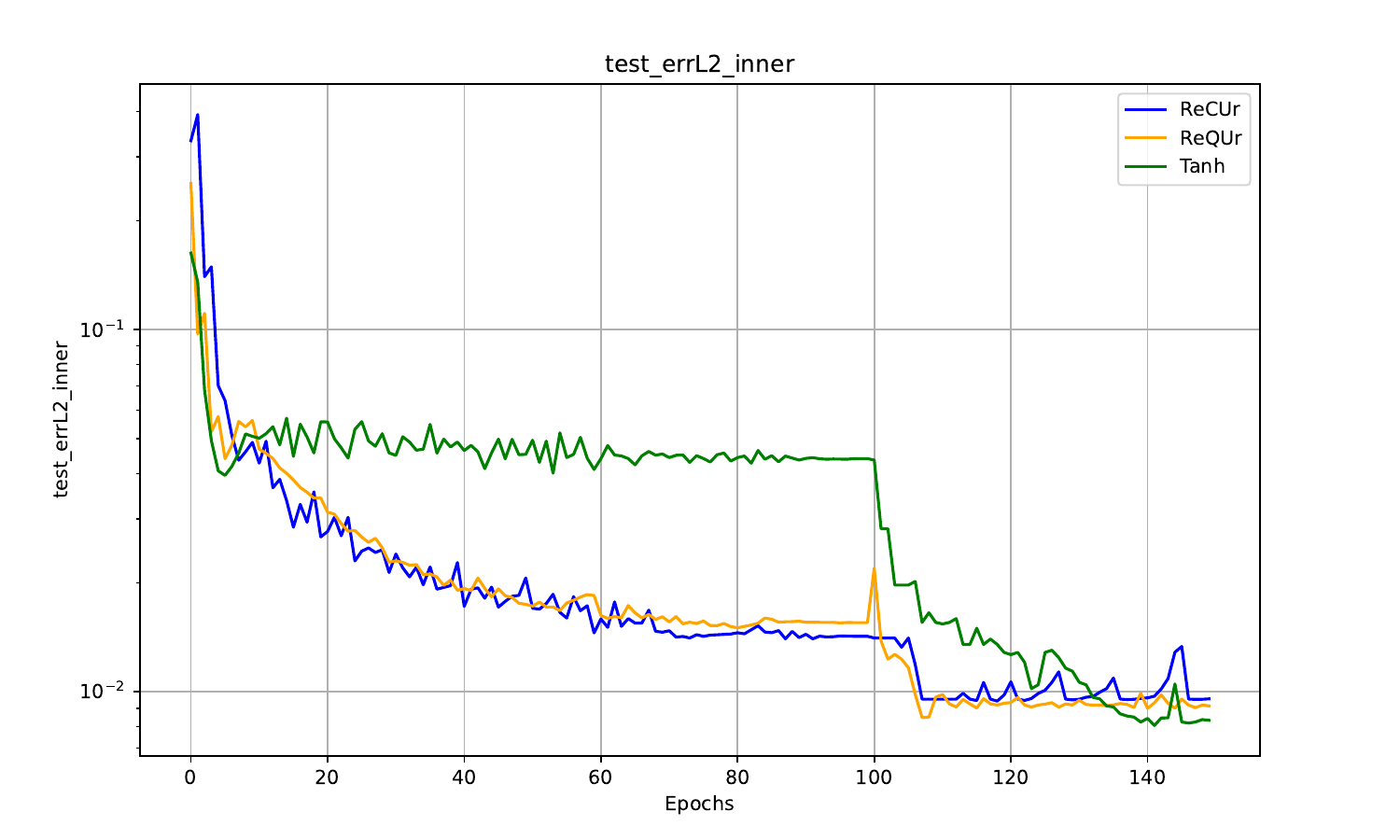}}
\subfloat[NatDRM: boundary error]{\label{fig:ex2NDBerrorb}\includegraphics[width=0.49\textwidth]{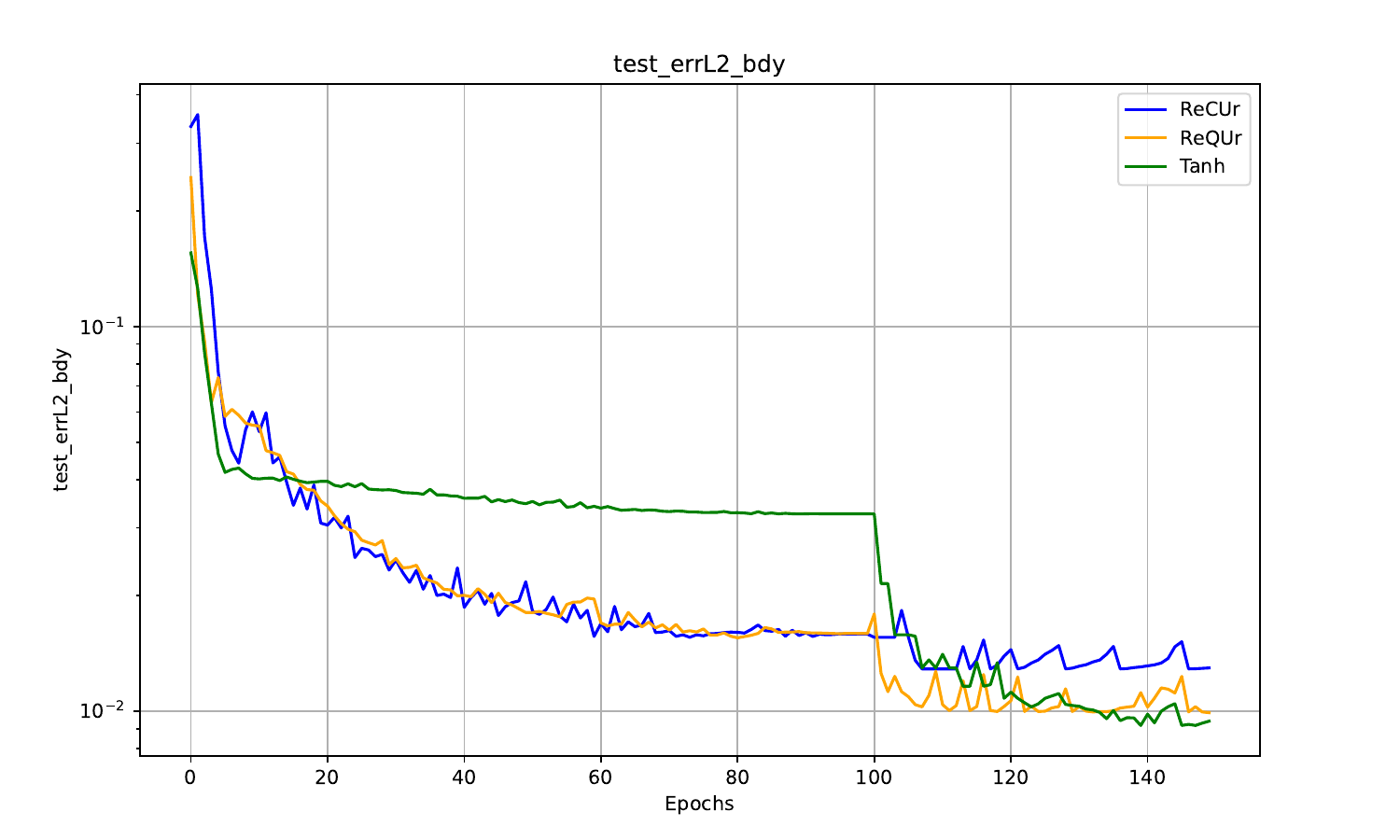}}
\\
\subfloat[DRM: $beta=100$ inner error]{\label{fig:ex2DRMerrori}\includegraphics[width=0.49\textwidth]{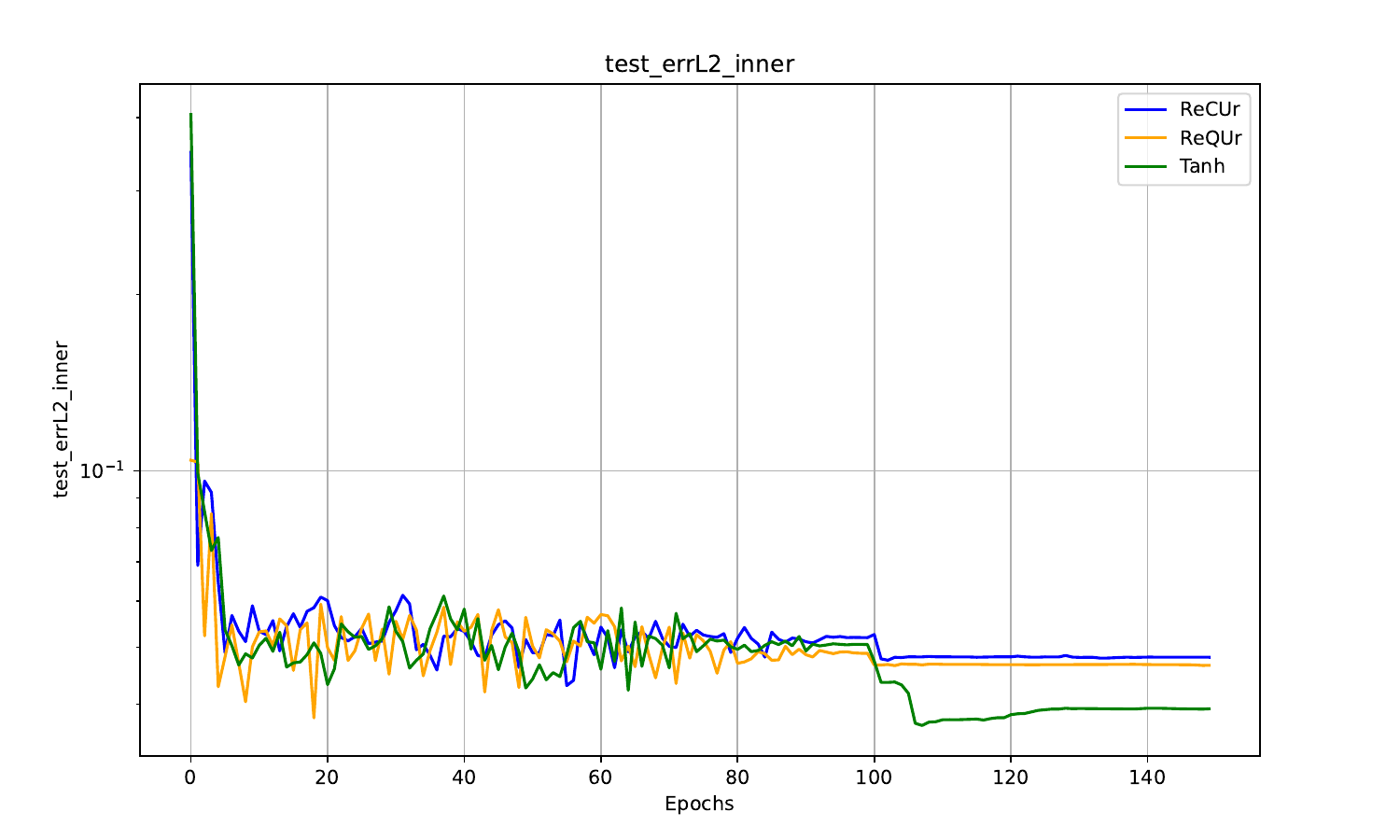}}
\subfloat[DRM: $beta=100$ boundary error]{\label{fig:ex2DRMerrorb}\includegraphics[width=0.49\textwidth]{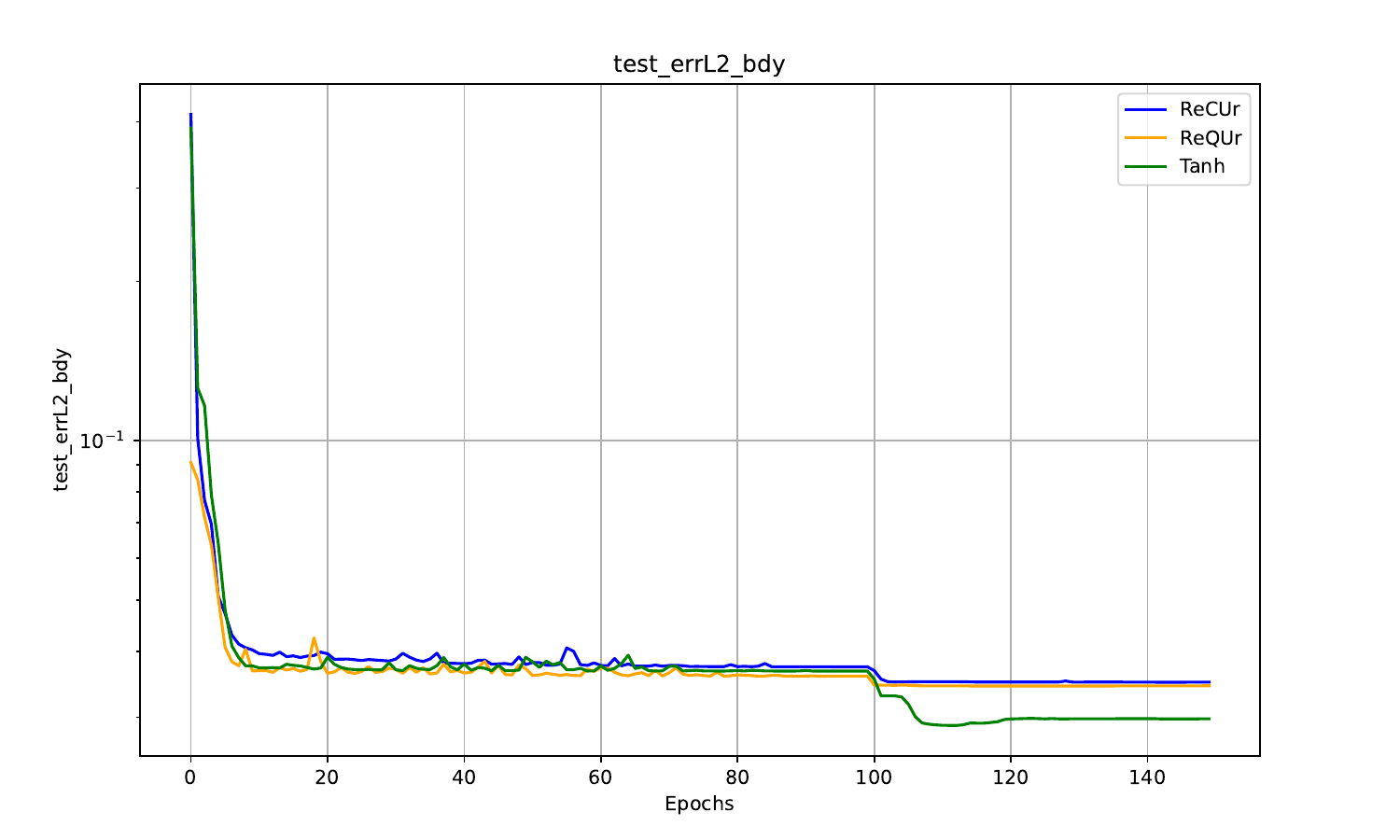}}
\caption{Relative $L^2$ interior and boundary errors for NatDRM and DRM ($\beta=100$) during training on the 6D benchmark \eqref{25}.}
\label{fig3}
\end{figure}

From these figures, the following conclusions can be drawn:

(1) $\mathbf{Robustness \ of \ activation \ functions.}$ Within the NatDRM framework, all three activation functions converge stably. Among them, ReCUr yields the lowest final error in most cases, and this advantage is particularly prominent in the 4D example (\Cref{fig2}). This is because ReCUr possesses a smooth Hessian together with higher-order function smoothness, providing reliable second‑order information for L‑BFGS and matching well with the regularity required by the curl and gradient operations in the variational formulation. The convergence behavior of ReQUr shows different characteristics: its first derivative is continuous, but its second derivative is discontinuous at breakpoints, leading to a discontinuous Hessian. This structure makes the optimization slightly inferior to ReCUr, yet after sufficient training, the final accuracy of the two activations remains on the same order. Tanh also has a smooth Hessian and thus exhibits a relatively stable convergence process; however, it is prone to vanishing gradients in saturation regions, which somewhat limits its global approximation capability in high-dimensional spaces.

(2) $\mathbf{Synchronized \ convergence \ of \ NatDRM.}$ A key qualitative difference is that in NatDRM the interior and boundary errors decrease almost synchronously (see \Cref{fig1}-\Cref{fig3} (a) and (b)). This indicates that the decomposition strategy effectively encodes boundary information into the auxiliary potentials \(\tilde{u}\) and \(\varphi\), allowing the interior and boundary solutions to be optimized cooperatively during training. In contrast, for DRM even with a moderate \(\beta=100\) as shown in the figures, the boundary error drops extremely fast while the interior error converges slowly or even plateaus. This is a typical symptom of the penalty term dominating the loss function, causing the optimization to overly favor boundary fitting and neglect interior accuracy.

(3) $\mathbf{Training \ stability \ across \ dimensions.}$ The training curves for the three test cases show that NatDRM achieves stable convergence within 150 epochs, with smooth and monotonic error decay. DRM, even with a moderately chosen \(\beta=100\), exhibits a relatively high plateau in the interior error after fast boundary convergence, and the interior error remains about one order of magnitude higher than that of NatDRM. For PINN, although both interior and boundary errors can eventually converge (see \Cref{tab:1} for accuracy data), the training process often exhibits violent oscillations, especially in high dimensions, indicating greater sensitivity to learning rate and optimizer settings. In all tested dimensions, NatDRM offers the most robust and trouble‑free training behavior.

\subsubsection{Visualization of Numerical Solutions and Error Distributions}

To further assess the solution quality of NatDRM on problems with varying complexity, we visualize numerical results with dimension-adaptive slice plotting strategies. For the three-dimensional case, stacked translucent slices are adopted to intuitively exhibit the full three-dimensional field structure of numerical solutions and corresponding pointwise errors; for the 4D and 6D high-dimensional problems, we present two‑dimensional cross‑sections of the numerical solutions obtained by NatDRM and DRM together with the exact solution, as well as the pointwise \(L^2\) error distribution of NatDRM. \Cref{fig:4} corresponds to the 3D variable‑coefficient problem \eqref{26}, \cref{fig:5} to the 4D semilinear problem \eqref{27}, and \cref{fig:6} to the 6D separable problem \eqref{25}.

\begin{figure}[tbhp]
\centering
\subfloat[DRM]{\label{fig:ex3DRMSolution}\includegraphics[width=0.49\textwidth]{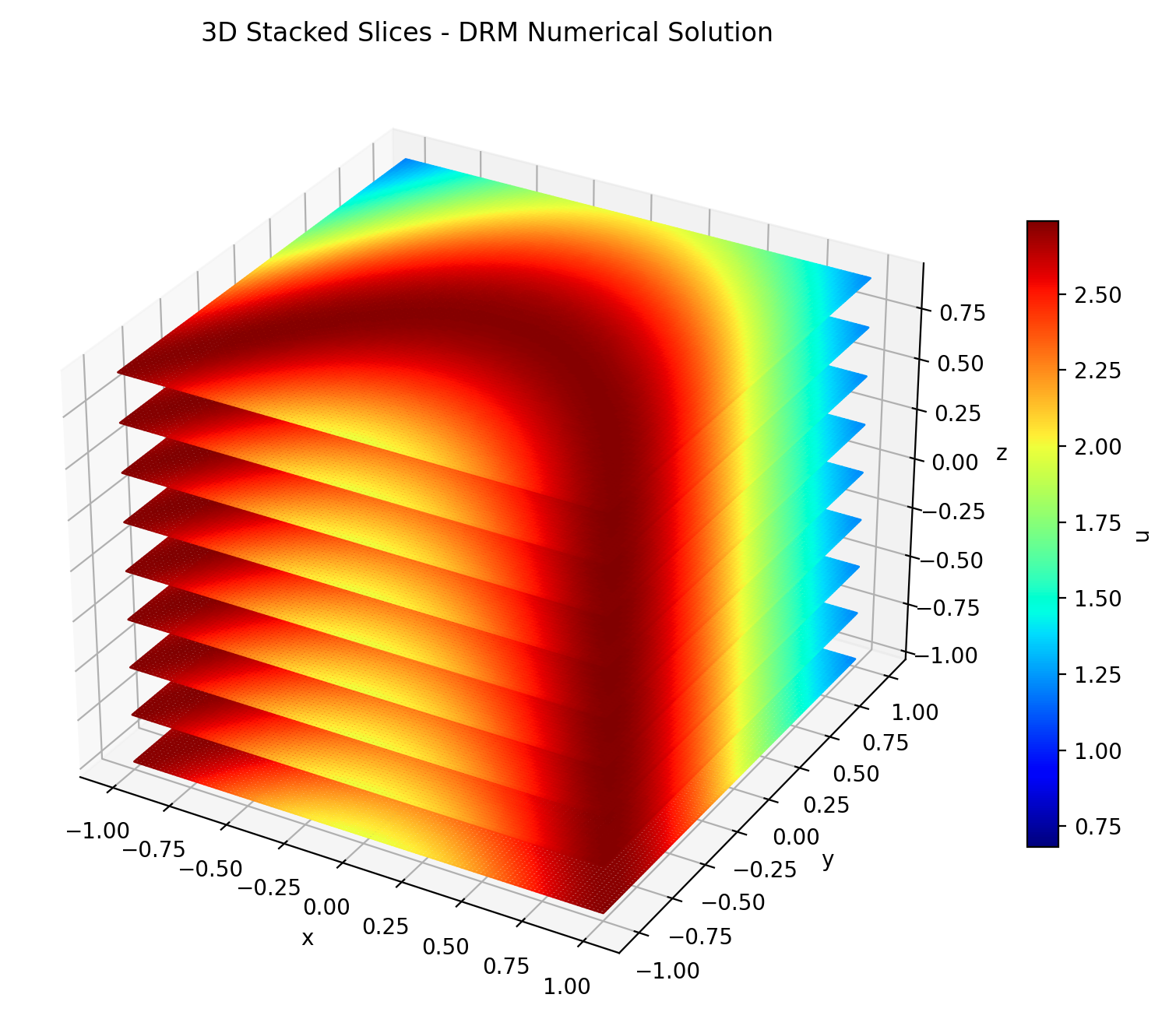}}
\subfloat[$L^2$ error distribution of DRM]{\label{fig:ex3DRMError}\includegraphics[width=0.49\textwidth]{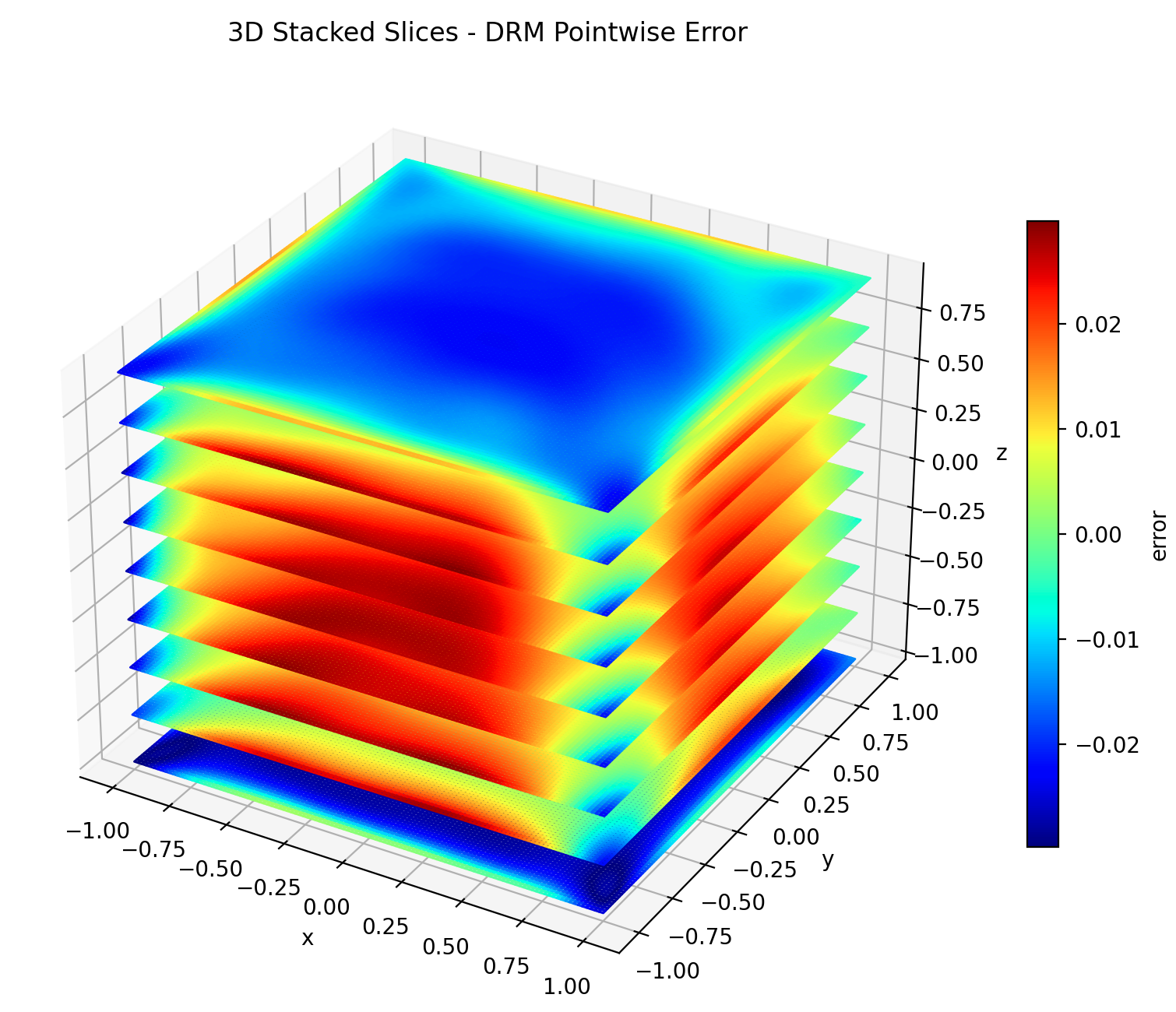}}
\\
\subfloat[NatDRM]{\label{fig:ex3NDBSolution}\includegraphics[width=0.49\textwidth]{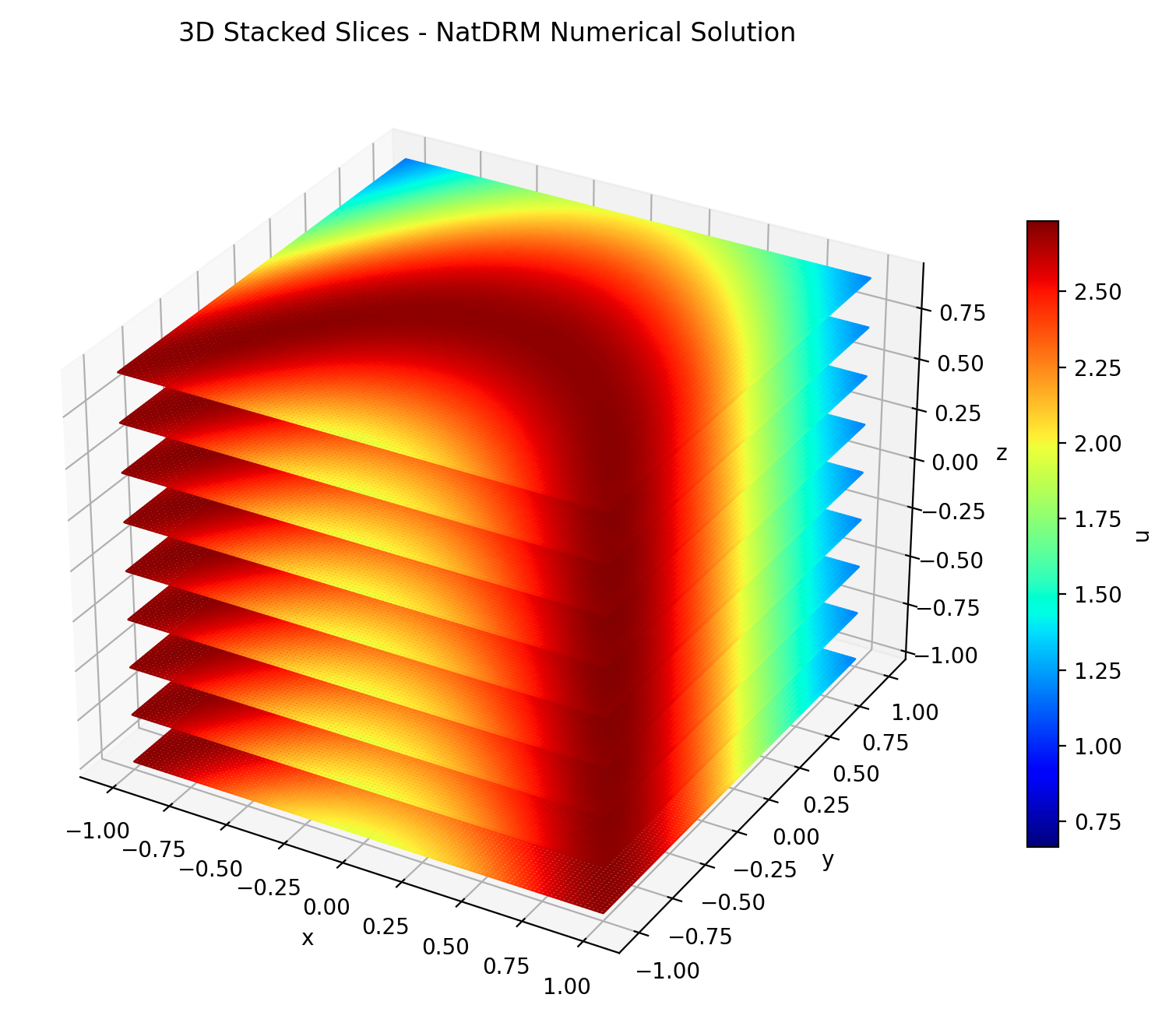}}
\subfloat[$L^2$ error distribution of NatDRM]{\label{fig:ex3NDBError}\includegraphics[width=0.49\textwidth]{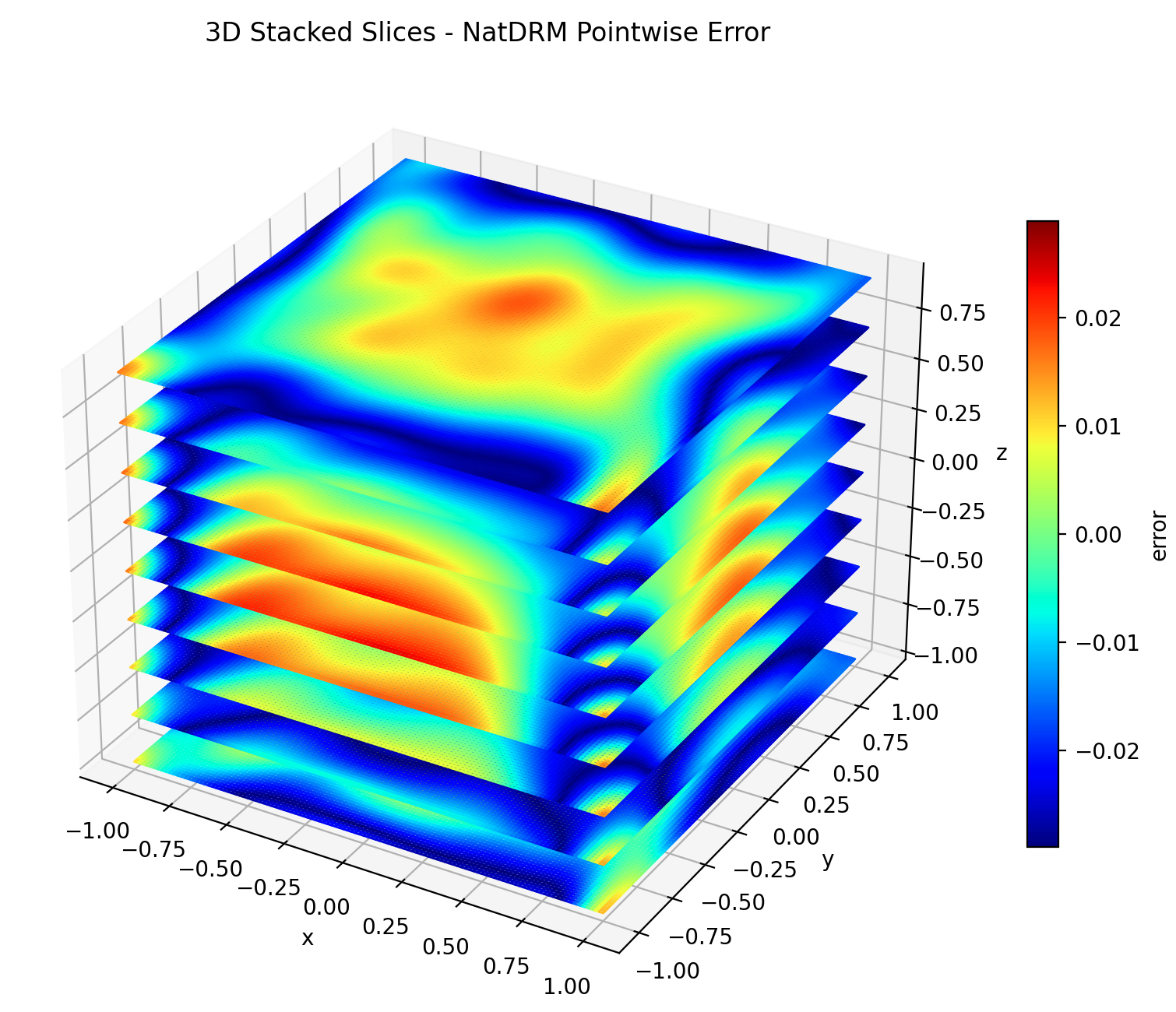}}
\caption{Numerical solutions of the 3D problem \eqref{26} obtained by NatDRM and DRM, and their corresponding $L^2$ error distributions.}
\label{fig:4}
\end{figure}

\begin{figure}[tbhp]
\centering
\subfloat[NatDRM]{\label{fig:ex4NDBSolution}\includegraphics[width=0.49\textwidth]{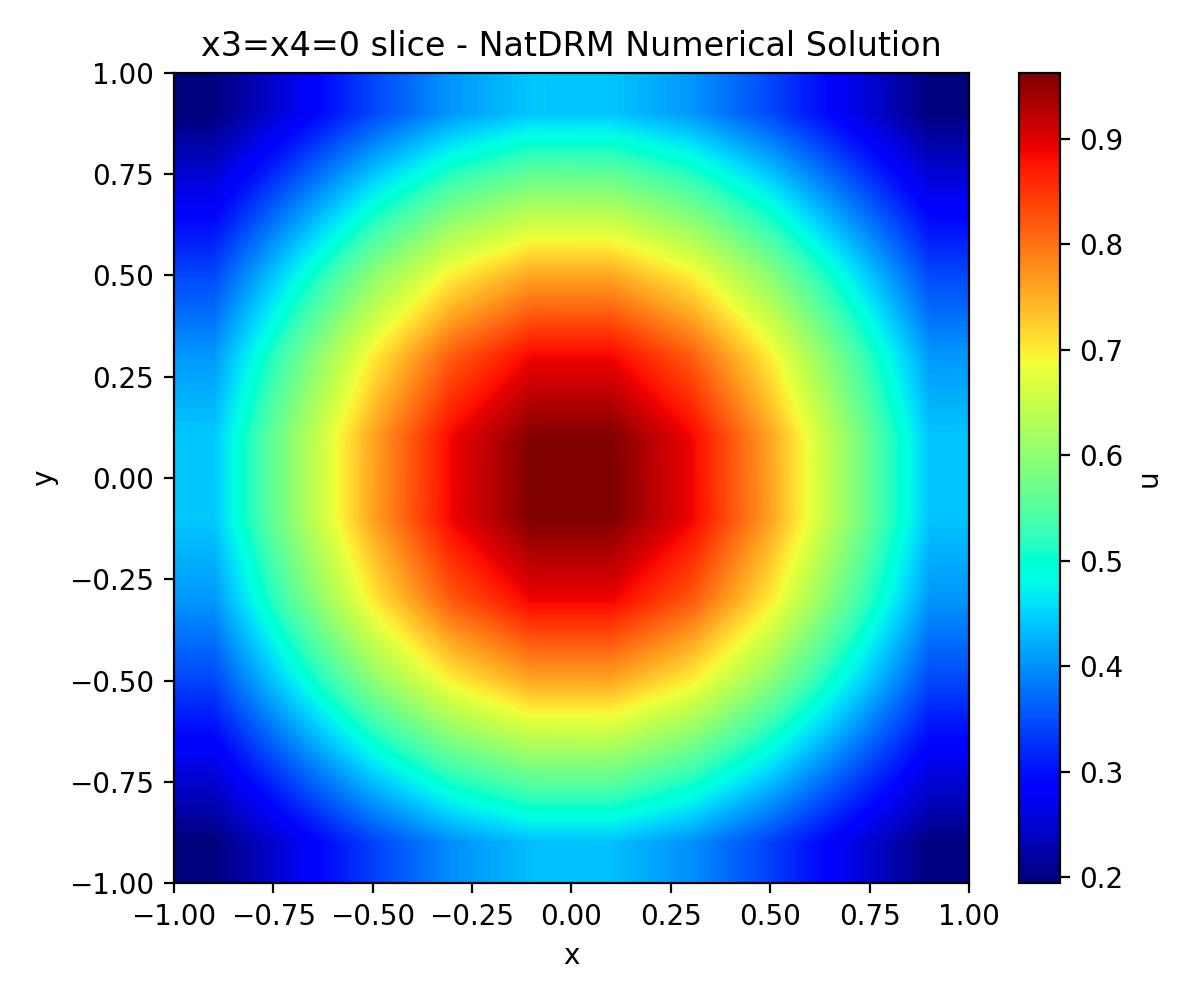}}
\subfloat[DRM]{\label{fig:ex4DRMSolution}\includegraphics[width=0.49\textwidth]{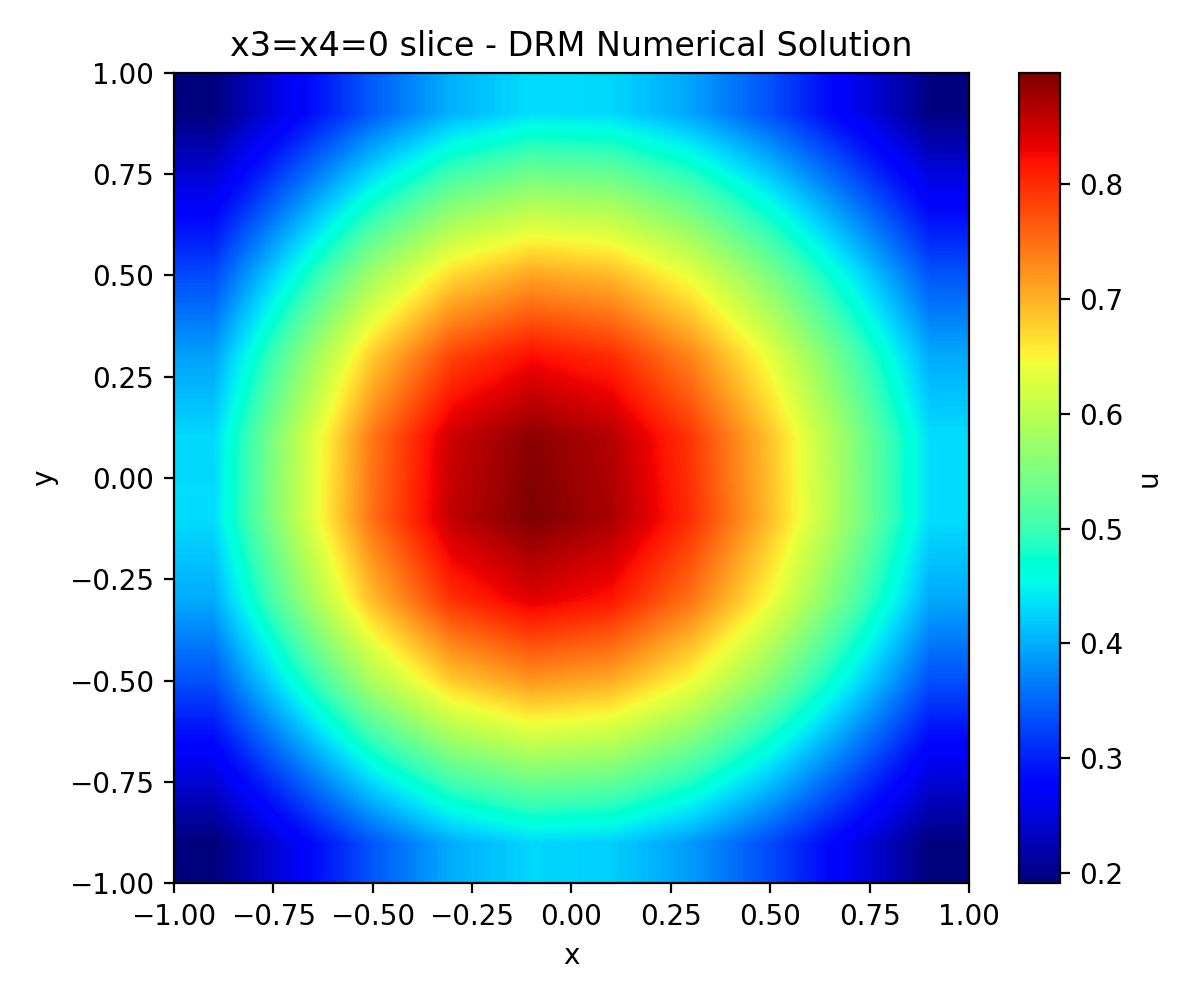}}
\\
\subfloat[Exact Solution]{\label{fig:ex4ExactSolution}\includegraphics[width=0.49\textwidth]{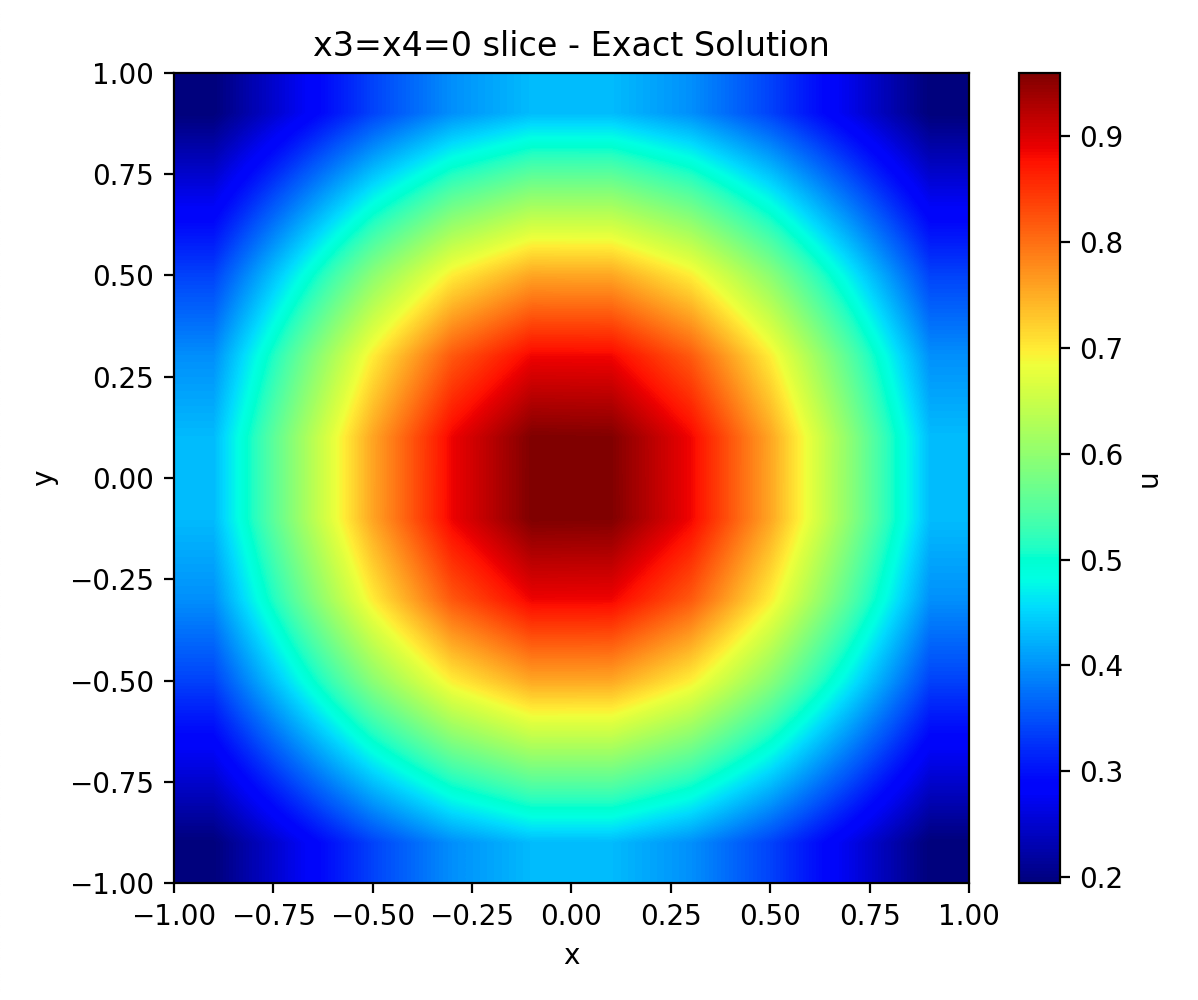}}
\subfloat[$L^2$ error distribution of NatDRM]{\label{fig:ex4NDBError}\includegraphics[width=0.49\textwidth]{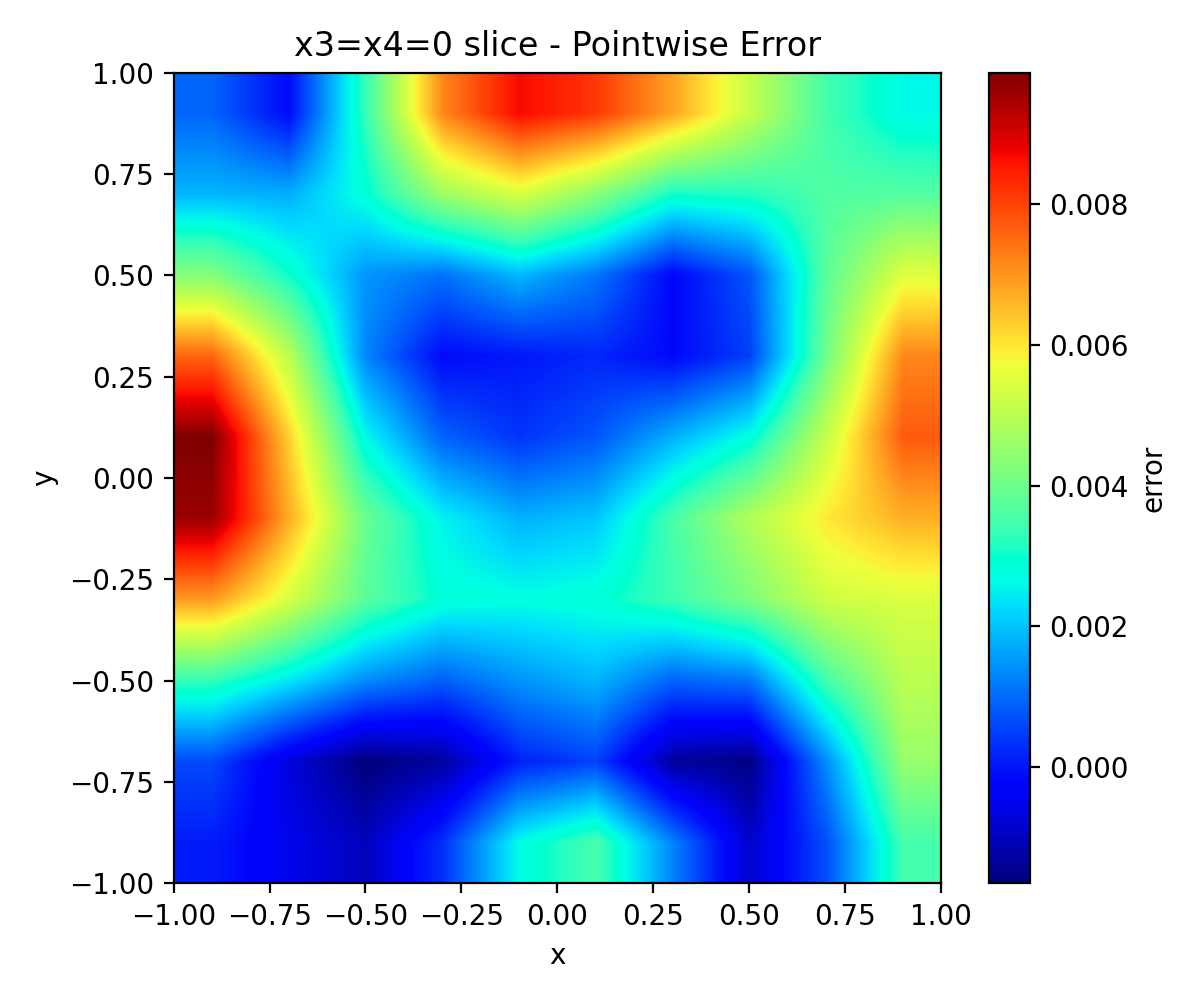}}
\caption{The distribution of numerical solutions of \eqref{27} obtained by different methods and $L^2$ error distribution of NatDRM in 4D.}
\label{fig:5}
\end{figure}

\begin{figure}[tbhp]
\centering
\subfloat[NatDRM]{\label{fig:ex2NDBSolution}\includegraphics[width=0.49\textwidth]{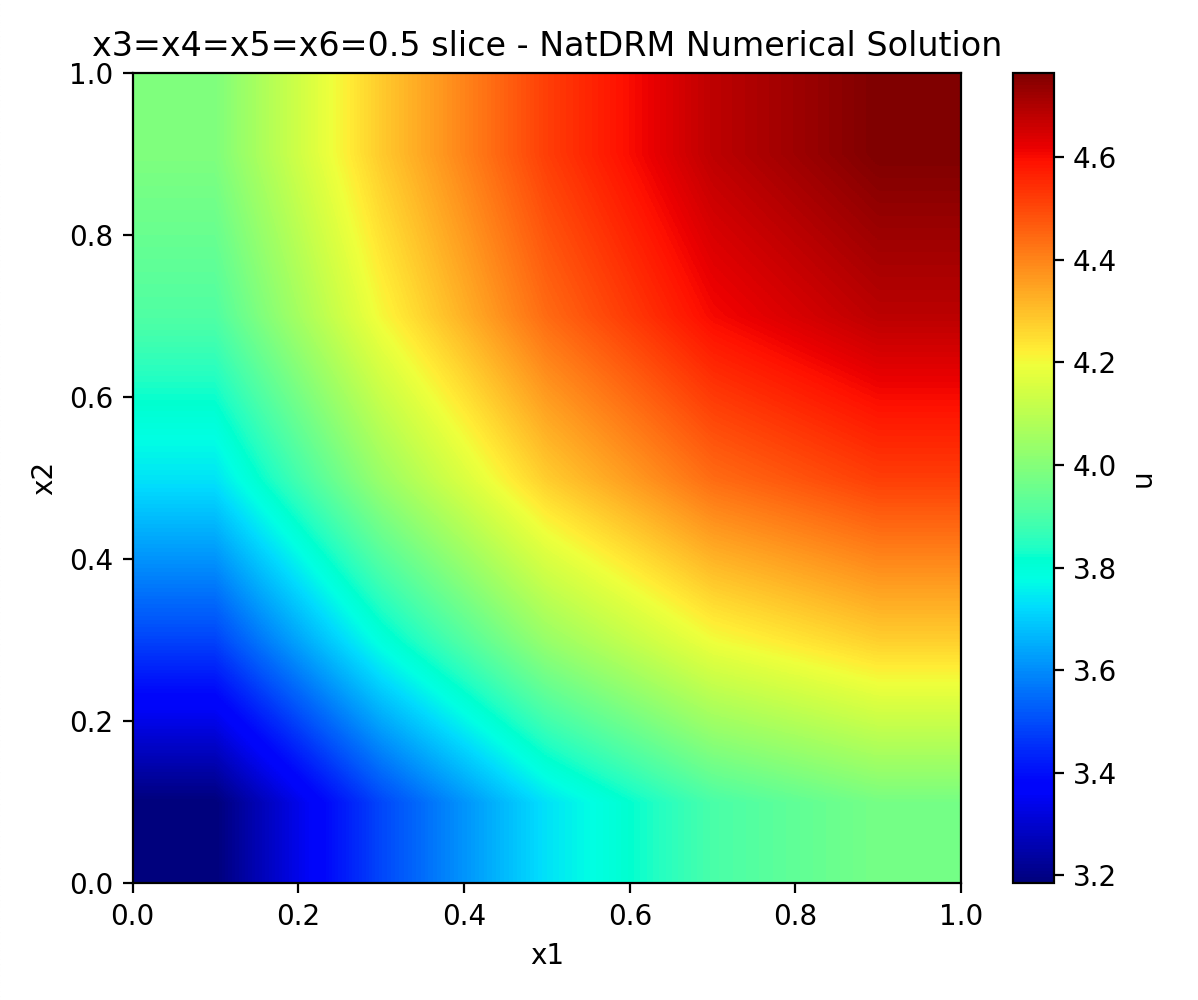}}
\subfloat[DRM]{\label{fig:ex2DRMSolution}\includegraphics[width=0.49\textwidth]{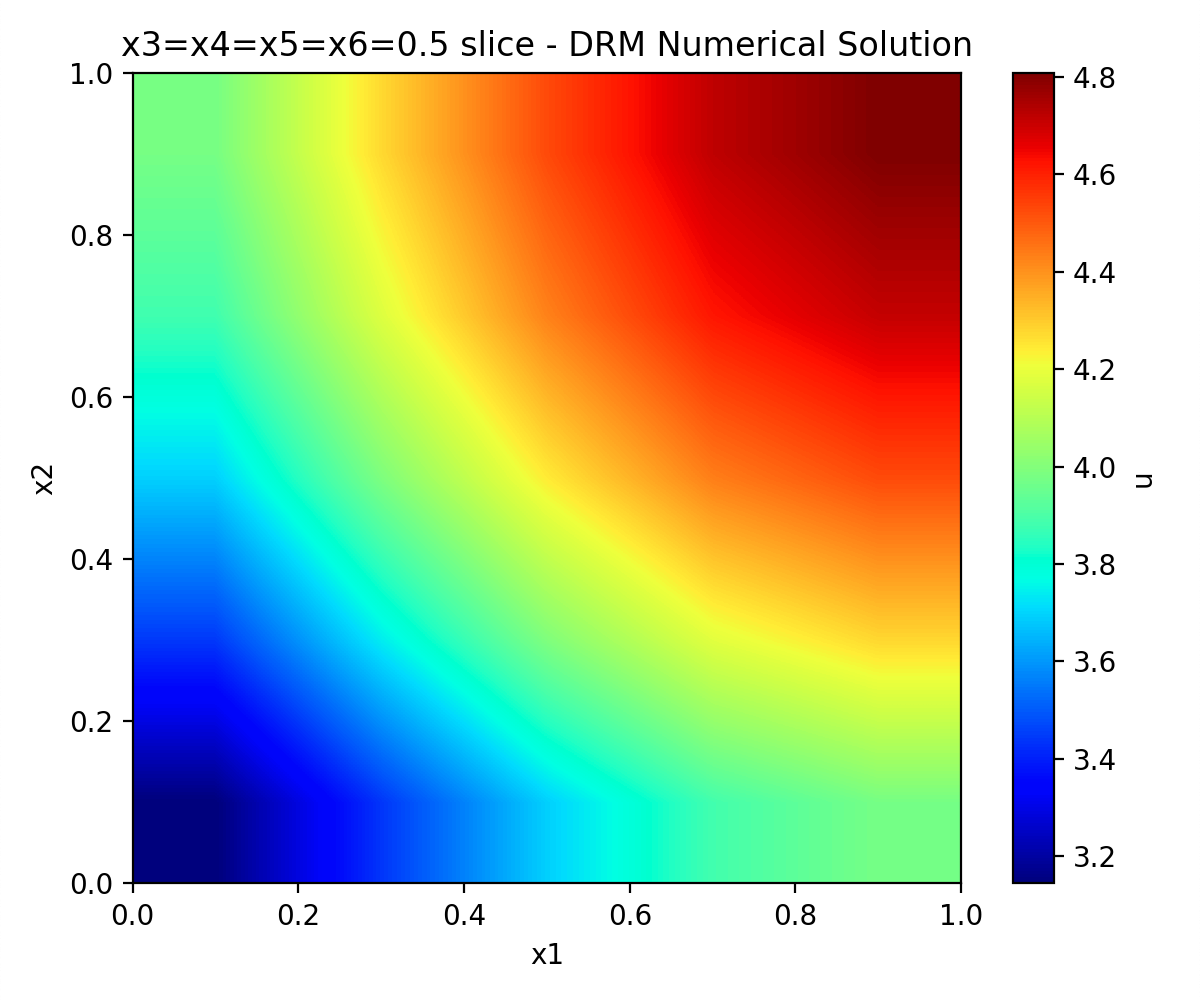}}
\\
\subfloat[Exact Solution]{\label{fig:ex2ExactSolution}\includegraphics[width=0.49\textwidth]{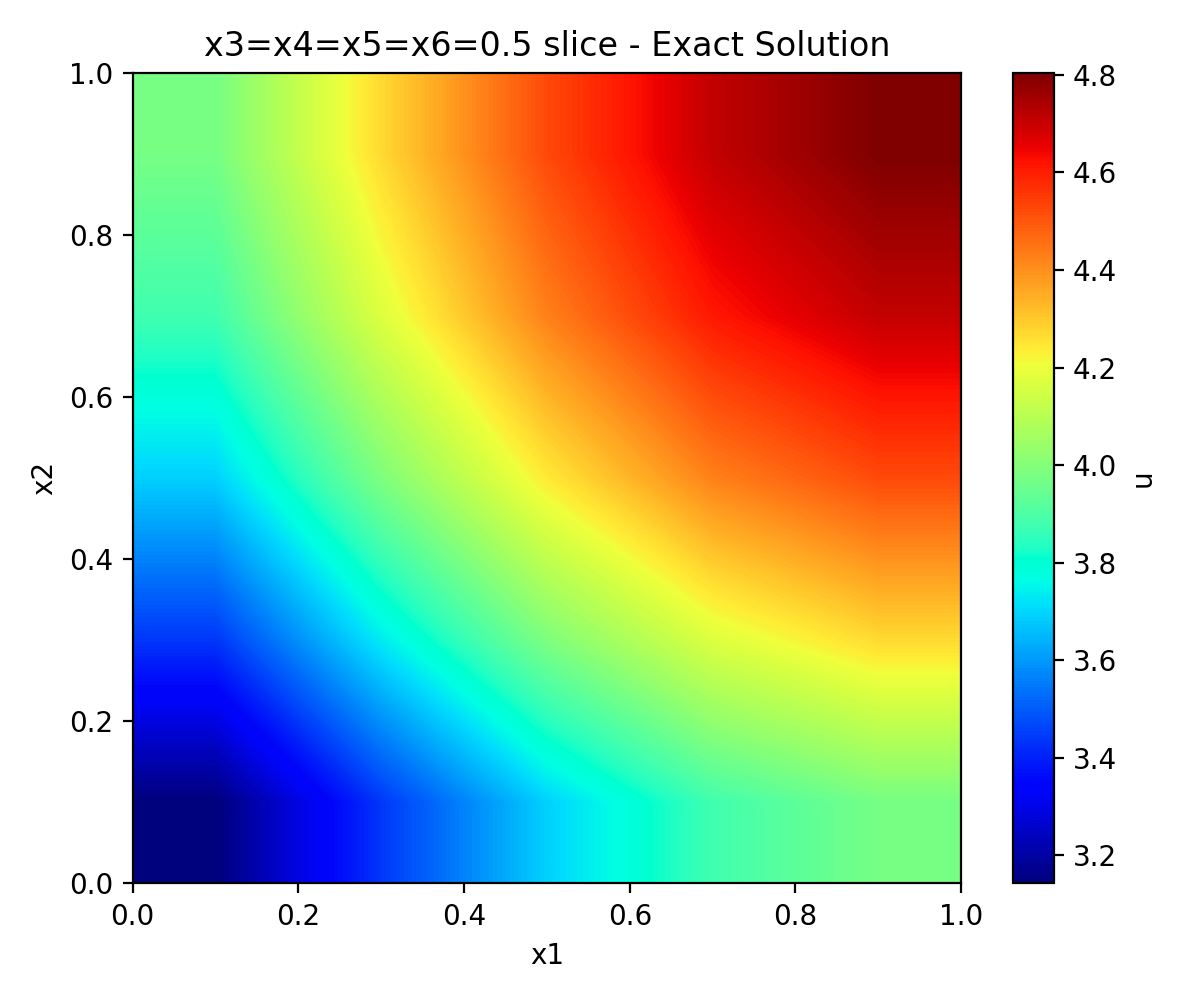}}
\subfloat[$L^2$ error distribution of NatDRM]{\label{fig:ex2NDBError}\includegraphics[width=0.49\textwidth]{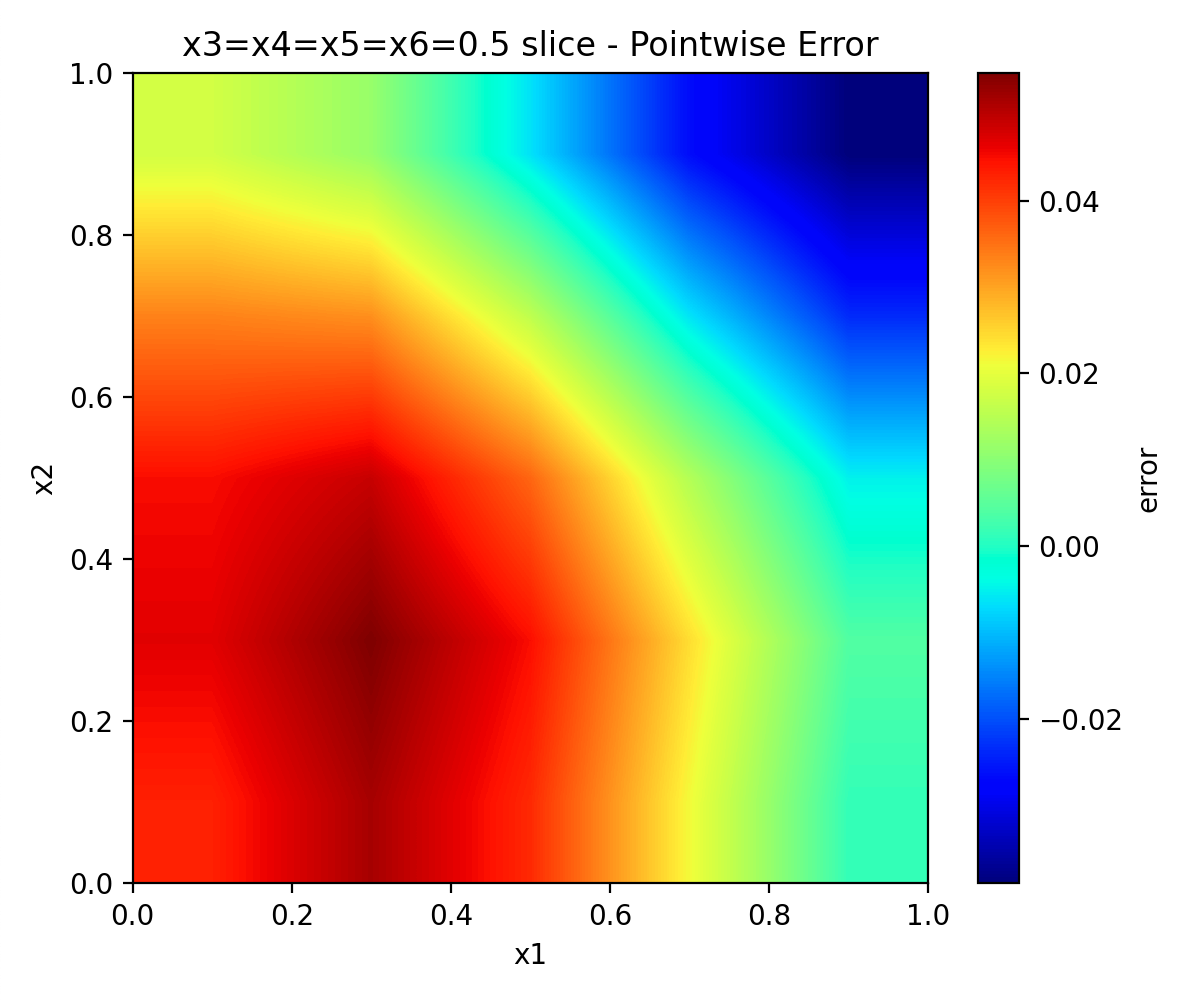}}
\caption{The distribution of numerical solutions of \eqref{25} obtained by different methods and $L^2$ error distribution of NatDRM in 6D.}
\label{fig:6}
\end{figure}
From these visualizations we observe the following:

NatDRM consistently reproduces the exact solution with high fidelity. In all three cases, the reconstructed solution of NatDRM matches the exact solution closely in both shape and magnitude. The \(L^2\) error distributions (panels (d) in \Cref{fig:4}-\Cref{fig:6}) show that errors are uniformly distributed between the boundary and the interior, with no significant concentration near the boundary or in regions with steep gradients, maintaining consistent accuracy throughout the entire computational domain. This demonstrates that NatDRM effectively handles variable coefficients \eqref{26} as well as semilinearity \eqref{27} without any special treatment.

DRM captures the main profile but exhibits visible boundary deviations. For the 3D variable‑coefficient problem (\Cref{fig:4}), DRM produces a solution that is generally correct in the interior, but noticeable discrepancies appear near the boundary layer. Comparing panels (b) and (d) of \Cref{fig:4}, it is evident that DRM has significantly larger boundary errors than NatDRM.
This is consistent with the penalty mechanism: a finite \(\beta=100\) cannot enforce the Dirichlet condition exactly, leading to a persistent boundary error that propagates slightly inward.

High‑dimensional applicability of NatDRM is further validated in the tested 6D setting (\Cref{fig:6}): cross‑sections match the exact solution well under our quadrature and network configuration. The semilinear term in \eqref{27} and the variable coefficient in \eqref{26} did not require penalty tuning in our experiments, using the extensions in \cref{sec:3-extensions}.

Integrating the numerical experiments from 3D to 6D, we draw the following conclusions:

(1) $\mathbf{Parameter‑free \ nature:}$ The proposed NatDRM successfully extends the penalty‑free strategy to high dimensions, achieving stable convergence in all tests without any parameter tuning, thereby avoiding the trial‑and‑error search for penalty parameters.

(2) $\mathbf{Accuracy \ and \ robustness:}$ In the vast majority of cases, the accuracy of NatDRM reaches or exceeds that of optimally tuned DRM and PINN, while exhibiting stronger robustness to the choice of activation functions and optimizers.

(3) $\mathbf{Computational \ efficiency:}$ While the proposed NatDRM exhibits a longer single-run time, it eliminates the tuning overhead of multiple trials, making its total time competitive; further performance improvements can be achieved through subproblem parallelization.

(4) $\mathbf{High‑dimensional \ applicability:}$ NatDRM maintains stable convergence and acceptable accuracy in six‑dimensional spaces, laying a foundation for tackling even higher‑dimensional problems. The successful treatment of variable‑coefficient and semilinear equations further demonstrates its versatility.

\section{Concluding remarks}
\label{sec:5}

In this paper, we have developed a unified framework that systematically extends the Natural Deep Ritz Method to high‑dimensional Poisson equations with Dirichlet boundary conditions. By exploiting the de Rham complex and its dual structure, the original essential boundary value problem is decomposed into three purely natural boundary value subproblems, thereby completely avoiding the use of boundary penalty terms. Based on this construction, we have presented loss functions suitable for any dimension \(d \geq 2\), enabling direct deployment in three‑, four‑, and even six‑dimensional spaces. Comprehensive numerical experiments on 3D, 4D, and 6D benchmarks, including constant‑coefficient Poisson, variable‑coefficient elliptic, and semilinear Poisson equations, fully validate the effectiveness of the proposed method.

We emphasize again that the primary advantage of our method lies in its complete elimination of boundary penalty parameter tuning. Standard DRM and PINN require trial‑and‑error selection of the penalty coefficient \(\beta\), and the optimal value varies depending on the specific problem, optimizer, and network architecture. For instance, in the 6D case, DRM with most \(\beta\) values fails to yield meaningful accuracy and converges only for specific \(\beta\), whereas NatDRM converges stably without any manual tuning and attains accuracy that is in the same order of magnitude as the optimally tuned performance of the other two methods. This parameter‑free property greatly simplifies the training pipeline and avoids the pitfalls of excessive stiffness from large penalties or insufficient boundary enforcement from small penalties.

Thanks to the subproblem decomposition strategy, NatDRM exhibits almost synchronous decay of interior and boundary errors during training, resulting in a well‑balanced and stable optimization process that is more robust to the choice of activation functions. The ReCUr activation consistently delivers the best performance, and its advantage becomes even more prominent in high dimensions. In terms of computational efficiency, although the single‑run time of NatDRM is higher than that of a single run of DRM or PINN due to the need to train three subnetworks simultaneously, the total cost becomes clearly competitive when accounting for the multiple trial runs typically required by DRM and PINN to locate a suitable penalty parameter. This advantage is further amplified as the dimensionality and model size increase, and explicit parallelization of the subproblems can substantially reduce the single‑run time.

Beyond constant‑coefficient Poisson equations, the numerical experiments on the variable‑coefficient problem and the semilinear problem demonstrate that NatDRM readily handles more complex differential operators and nonlinearities without any modification to the core penalty‑free decomposition. The visualizations in the figures confirm that the solution quality remains high, with errors distributed relatively uniformly across the domain. This demonstrates that NatDRM applied to six‑dimensional problems does not suffer from excessive boundary errors, and the interior approximation remains accurate. This versatility suggests that NatDRM can serve as a general‑purpose solver for a wide class of essential boundary value problems.

Building on the unified penalty-free NatDRM framework validated for 3D–6D problems, three key directions are pursued to enhance its scalability and generality:

(1) Scalable computation and  benchmarking: We will develop a distributed parallel training architecture leveraging the natural decoupling of the three subproblems, conduct systematic dimension-scaling experiments up to 20D, and establish a fair total-cost evaluation framework for rigorous comparison with alternative boundary treatments.

(2) Efficient high-dimensional potential representation: We will investigate structured parameterization techniques, including tensor decompositions of low rank and sparse representations that incorporate antisymmetry constraints, to reduce the computational burden of potentials of tensor type without sacrificing the accuracy of numerical solutions.

(3) Generalization to broader PDE classes: We will extend the natural boundary decomposition framework to handle multiscale problems with discontinuous coefficients/interface conditions and non-self-adjoint operators, and develop advanced nonlinear solvers for strongly nonlinear problems.




\bibliographystyle{siamplain}
\bibliography{references}
\end{document}

%% file: ex_shared.tex
\usepackage{lipsum}
\usepackage{amsfonts}
\usepackage{graphicx}
\usepackage{epstopdf}
\usepackage{algorithmic}
\ifpdf
  \DeclareGraphicsExtensions{.eps,.pdf,.png,.jpg}
\else
  \DeclareGraphicsExtensions{.eps}
\fi

\newsiamremark{remark}{Remark}
\newsiamremark{hypothesis}{Hypothesis}
\crefname{hypothesis}{Hypothesis}{Hypotheses}
\newsiamthm{claim}{Claim}
\newsiamremark{fact}{Fact}
\crefname{fact}{Fact}{Facts}

\headers{Penalty-Free Natural Deep Ritz Method}{J. Chen, X. Ji, H. Yu, and S. Zhang}

\title{Penalty-Free Natural Deep Ritz Method Based on \\ de Rham Complex for High-Dimensional Dirichlet Boundary Value Problems\thanks{Submitted to the editors DATE.
\funding{This work was partially supported by the National Natural Science Foundation of China (NSFC) under grant numbers 92370205, 12494543, 12271512, and 12371389, and by the Strategic Priority Research Program of the Chinese Academy of Sciences under grant numbers XDA0480504 and XDB0640000.}}}

\author{Jiarong Chen\thanks{School of Mathematics and Statistics, Beijing Institute of Technology, 100081, Beijing, China
  (\email{3120241504@bit.edu.cn}).}
\and Xia Ji\footnotemark[2] \thanks{Beijing Key Laboratory on MCAACI, Beijing Institute of Technology, 100081, Beijing, China (\email{jixia@bit.edu.cn}).}
\and Haijun Yu\footnotemark[5] \thanks{State Key Laboratory of Mathematical Sciences (SKLMS) and State Key Laboratory of Scientific and Engineering Computing (LSEC), Institute of Computational Mathematics and Scientific/Engineering Computing, Academy of Mathematics and Systems Science, Chinese Academy of Sciences, 100190, Beijing, China (\email{hyu@lsec.cc.ac.cn}).}
\and Shuo Zhang\footnotemark[4] \thanks{School of Mathematical Sciences, University of Chinese Academy of Sciences, 100049, Beijing, China (\email{szhang@lsec.cc.ac.cn}).} 
}

\usepackage{amsopn}
